\documentclass[a4paper]{amsart}
\usepackage{amssymb}
\usepackage{tikz}
\usepackage{epic,curves,mfpic}
\usepackage{hyperref}

\newtheorem{theo}{Theorem}[section]
\newtheorem{lemma}[theo]{Lemma}

\newtheorem{prop}[theo]{Proposition}

\theoremstyle{definition}

\theoremstyle{remark}
\newtheorem{exam}[theo]{Example}

\begin{document}

\title{Structure trees and   networks}
\author{M.J.Dunwoody}

\begin{abstract} 

In this paper it is shown that for any network there is a uniquely determined network based on a {\it structure tree} that provides a convenient way of determining  a minimal cut  separating a pair $s, t$ where each of $s, t$ is either a vertex or an end in the original network.
A Max-Flow Min-Cut Theorem is proved for any network.    In the case of a Cayley Graph for a finitely generated group the theory provides another 
proof of Stallings' Theorem on the structure of groups with more than one end.

\end{abstract}
\maketitle

\newcommand{\Z}{\mathbb{Z}}
\newcommand{\N}{\mathbb{N}}
\newcommand{\R}{\mathbb{R}}
\newcommand{\Q}{\mathbb{Q}}
\newcommand{\cO}{\mathcal{O}}

\newcommand{\B}{\mathcal{B}}
\newcommand{\A}{\mathcal{A}}
\newcommand{\T}{\mathcal{T}}
\newcommand{\G}{\mathcal{G}}
\newcommand{\bv}{{\bf v}}
\newcommand{\bu}{{\bf u}}
\newcommand{\bp}{{\bf p}}
\newcommand{\bz}{{\bf z}}
\newcommand{\C}{\mathcal{C}}

\newcommand{\bP}{\mathcal{P}}
\newcommand{\V}{\mathcal{V}}
\newcommand{\ci}{\mathcal{I}}
\newcommand{\h}{\mathcal{H}}
\newcommand{\ca}{\mathcal{A}}
\newcommand{\cm}{\mathcal{M}}
\newcommand{\cs}{\mathcal{S}}
\newcommand{\m}{\mu}
\newcommand{\ce}{\mathcal{E}}
\newcommand{\cR}{\mathcal{R}}

\newcommand{\cz}{\mathcal{Z}}
\newcommand{\cb}{\mathcal{B}}

\newcommand{\bG}{\bar G}
\newcommand{\bH}{\bar H}
\renewcommand{\d}{{\rm d}}

\newcounter{fig}
\setcounter{fig}{0}

\section{Introduction}
In this paper a way of extending the theory of finite networks to networks based on arbitrary graphs is presented.    
Results for finite networks such as the Max-Fflow Min-Cut Theorem (MFMC) and the existence of a Gomory-Hu Tree are shown to be special
cases of our results for more general networks.   It is also the case that Stallings' Theorem on the structure of groups with more than one end also
follows from the theory developed here.     It is very pleasing (to me at least) that there is a theory that includes both the Stallings' Theorem and the MFMC.

In his breakthrough work  \cite {[St]}  on groups with more than one end,  Stallings showed that a finitely generated group has a Cayley graph (corresponding to a finite generating set) with more than one end if
and only if it has a certain structure.   At about that time Bass and Serre (see \cite {[DD]} or \cite {[S]})  developed their theory of groups acting on trees and it was clear that
the structure of  a  group with more than one end, as  in Stallings'  Theorem,  was associated with an action on a tree.   In \cite {[D1]} I gave a proof of Stallings' result by constructing
a tree on which the relevant group acted.   This involved showing that if the finitely generated  group $G$  had more than one end, then there is a subset $B \subset G$ such that both $B$ and $B^*$ are infinite,  $\delta B$ is finite, and the set $\ce = \{ gB | g \in G\} $ is almost nested.   A set $\ce $ of cuts is almost nested if for every $A, B \in \ce$ at least one corner of $A$ and $B$ is finite.   A corner of $A, B$ is one of the four sets 
$A\cap B, A^*\cap B, A\cap B^*, A^*\cap B^*$,  where $A^*$ is the complement of $A$.

In \cite {[D2]} I gave a stronger result by showing that if a group $G$ acts on a graph $X$ with more than one end, then there exists a subset 
$B \in \B X$ such that $B$ and $B^*$ are both infinite and for any $g \in G$ the sets $B$ and $g B$ are nested, i.e. at least one of the four corners
is empty.    The set of all such $gB$ can 
be shown to be the edge set of a tree, called a structure tree.

This result was further extended by Warren Dicks and myself   \cite {[DD]}.  In Chapter II of that book it is shown  that for any graph $X$ the Boolean ring $\B X$
has a particular nested set of generators invariant under the automorphism group of $G$.   
    At the time I thought that the result when applied to  finite graphs was of little interest.
This was partly because an action of a group on a finite tree is always trivial, i.e. there is always a vertex of the tree fixed by the whole group.
This is not the case for groups acting on infinite trees: the theory of such actions is the subject matter of Bass-Serre theory.
Also for a finite graph $X$, there is always a nested set of generators for $\B X$ consisting of single elements subsets.
The belated realisation that the theory developed in \cite {[DD]} might be of some significance for finite networks occurred only recently.

In 2007 Bernhard Kr\" on asked me if one could develop a theory of structure trees for graphs with more that one vertex end rather than
more than one edge end.    These are connected graphs that have more than one infinite component after removing finitely many vertices.
We were able to develop such a theory in \cite {[DK]} .  In the course of  our work on this,  we realised that we could develop a theory of  structure trees 
for finite graphs that generalised the theory of Tutte \cite {[T]}, who obtained a structure tree result for $2$-connected finite graphs that are not $3$-connected.     The theory for vertex cuts is  more complicated than that for edge cuts.
In 2008 I learned about the cactus theorem for min-cuts from Panos Papasoglu.   This theory, due to Dinits, Karsanov and Lomonosov   \cite {[DKL]} (see also \cite {[FF]}) is for finite networks.
It is possible, with a bit more work, to deduce the cactus theorem  from the proof of  Theorem \ref {maintheorem} .
Evangelidou and Papasoglu  \cite {[EP]} have obtained a cactus theorem for edge cuts in infinite graphs, giving a new proof of Stallings' Theorem.
In \cite {[DW]} Diekert and Weiss gave a definition for thin cuts, which is equivalent to the one given in \cite {[DD]} (see Lemma \ref {thin}),
but which made more apparent the connection with the Max-Flow Min-Cut Theorem.  I also had a very helpful email exchange with Armin Weiss.
Weiss told me about Gomory-Hu trees that are structure trees in finite networks.  

Thinking about these matters finally led me  to think about structure trees for edge cuts in finite graphs and networks and the realisation
that the theory developed in \cite {[DD]} might be of some interest  when applied  to finite networks.

In Section 2 the theory for finite networks is recalled.  The theory is presented in such a way as to suggest the way it can be 
generalised to arbitrary networks.    This generalisation is obtained in Section 3.
For any network $N$ we obtain a canonically  determined sequence of trees $T_n$ that provide complete information
about the separation of a pair $s,t$ where each of $s$ and $t$ is either a vertex or an end of $X$.   
It is only possible to obtain all such information from a single tree $T_n$ if $X$ is {\it accessible}.   A graph is accessible if there is an integer
$n$ such that any two ends can be separated by removing at most $n$ edges.
   However there are locally finite vertex transitive
graphs that are inaccessible.  Such graphs are constructed in \cite {Dun} or \cite {Dunwoody1993}.

The situation for edge cuts contrasts with the situation for vertex cuts.   Thus there is a canonically determined sequence of trees that
separates a pair $s,t$ from the set of vertices or ends of the graph $X$.   For vertex cuts, one can only find a canonically defined structure tree
that separates a pair $\kappa $-inseparable sets or a pair of vertex ends,  where $\kappa $ is the smallest integer for which it is possible to separate
such a pair.

%In tis paper it is shown how the theory of finite networks can be generalised to networks based on an arbitrary graphs.

\section {Finite Networks}
In this section we define our terminology, but restrict attention to networks based on finite graphs.

We define a network  $N$ to be a finite, simple, connected graph  $X$  and a map $c : EX \rightarrow \{1, 2, \dots \}$.

Let $s, t \in VX$.   An {\it $(s, t)$-flow } in $N$ is a map  $f :  EX \rightarrow \{ 0, 1, 2, \dots \}$  together with an assignment of a direction to each edge  $e$ so that its vertices are $\iota e$ and $\tau e$  and the following holds.

\begin {itemize}
\item [(i)]   For each $e \in EX$,  $f(e) \leq c(e)$.

\item [(ii)] If we put $f^+ (v) =  \Sigma  ( f(e) | \iota e = v ) $ and $f^-(v) = \Sigma (f (e) | \tau e = v )$, then for every
$v \in VX, v \not= s, v \not= t$,  we have $f^+(v) = f^-(v)$. 
That is, at every vertex except $s$ or $ t$, the flow into that vertex is the same as the flow out.

\end {itemize}  

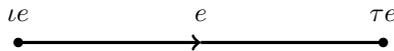
\begin{figure}[htbp]
\centering
\begin{tikzpicture}[scale=.8]
\draw [-> , very thick] (0,0) -- ( 3,0) ;
\draw [very thick] (3,0) --(6,0) ;
\filldraw (0,0) circle (2pt);
\filldraw (6,0) circle (2pt);
\draw (0,.2) node [above] {$\iota e$};
\draw (3,.2) node [above] {$e$};
\draw (6,.2) node [above] {$\tau e$};

\end{tikzpicture}
%\vskip-2mm\caption{}\label{fig:Arrow}\vskip-3mm
\end{figure}

It is easy to show that in an $(s,t)$-flow,    $f^+(s) - f^-(s)   =  -(f^+(t) - f^-(t))$.   The {\it value} of the flow is defined to be $|f| = |f^+(s) - f^-(s)|$.
We define a {\it cut} in $X$ to be a subset  $A$  of $VX$,  $A\not= \emptyset, A\not= VX$.   If $A$ is a cut then so is its complement $A^*$.
If $N$ is a network and $A \subset VX$ is a cut, then the capacity $c(A)$ of $A$ is the sum $c(A) = \Sigma \{ c(e) | e = (u,v), u \in A. v \in A^*\}$.
We define  $\delta A$ to be the set of edges with one vertex in $A$ and one in $A^*$, so that $c(A)$ is the sum of the values $c(e)$ as $e$ ranges over the edges of $\delta A$.    We could replace each edge $e$  of $X$ with $c(e)$ edges joining the same two vertices and then have a theory
in which the capacity of a cut is the number of edges in $\delta A$.

In Figure \ref{fig:Tutte} a network is shown, together with a max-flow (which has value $7$), together with a corresponding min-cut.

\begin {theo} [The Max-Flow Min-Cut Theorem \cite {[FF1]}]  The maximum value of an  $(s, t)$-flow is the minimal capacity of a cut separating $s$ and $t$.
\end {theo}

\begin{figure}[htbp]
\centering
\begin{tikzpicture}[scale=.8]
%\draw (0,18) node { };
%\draw (0,0) node { };

%\put (0,130){

\path (5,28) coordinate (p1); 
\path (6,31) coordinate (p2);  
\path (3,32) coordinate (p3); 
\path (3,30) coordinate (p4);  
\path (0,31) coordinate (p5);  
\path (1,28) coordinate (p6);  
\path (3,27) coordinate (p7);  

\path (9,26) coordinate (p17);  
\path (9,25) coordinate (nn);  
\path (8,29) coordinate (p18);   
\path (11,30) coordinate (p19);  
\path (12,27) coordinate (p20); 
\path (7,27) coordinate (p21); 
\path (9,31) coordinate (p22); 
\path (13,29) coordinate (p23); 
\path (10.8,25.7) coordinate (p24);
\path (11.5,23.5) coordinate (p25);

\draw [red] (10,30.5) node [above] {2};

\draw [red] (12,29.5) node [above] {5};

\draw [red] (8.5,30) node [left] {8};

\draw [red] (12.25,8) node [left] {3};

\draw [red] (9.5, 29.5) node [above] {5};

\draw [red] (11.5, 28.5) node [right] {1};

\draw [red] (7.7,27.7) node [above] {2};

\draw [red] (8.7,27.5) node [right] {1};

\draw [red] (7.9,26.5) node [above] {8};

\draw [red] (10.1,26.4) node [above] {2};

\draw [red] (6,7.25) node [above] {6};

\draw [red] (4,27) node [above] {5};

\draw [red] (2,27) node [above] {2};

\draw [red] (4,27) node [above] {5};

\draw [red] (2,29) node [below] {1};

\draw [red] (4,29) node [below] {3};

\draw [red] (7.5, 31) node [above] {6} ;

\draw [red] (6, 27) node [above] {6} ;

\draw [red] (3,30.5) node [right] {3};

\draw [red] (3,28) node [above] {1};

\draw [red] (5.5, 29.5) node [right] {1};

\draw [red] (4, 30) node [right] {1};

\draw [red] (1.5,31.5) node [above] {3};

\draw [red] (4.5, 31.5) node [above] {1};

\draw [red] (.5, 29.5) node [left] {6};

\draw [red] (2, 30) node [left] {4};

\filldraw (p2) circle (1pt);
\filldraw (p3) circle (1pt);
\filldraw (p4) circle (1pt);

\filldraw (p5) circle (1pt);
\filldraw (p7) circle (1pt);

\filldraw (p18) circle (1pt);
\filldraw (p19)circle (1pt);
\filldraw (p21) circle (1pt);

\filldraw (p22) circle (1pt);

\filldraw (p23) circle (1pt);

\draw (p1)--(p6) ;
\draw (p1) -- (p2) -- (p3)--(p1)--(p4)--(p3) -- (p5)--(p6)--(p3);
\draw (p4)--(p6) ;
\draw  (p1)--(p7)--(p1) ;
\draw (p7) -- (p6) ;

\draw (p17)--(p18)--(p19)--(p20)--(p17);
\draw (p1)--(p17) ;
\draw (p20)--(p23)--(p22)--(p22)--(p21);
\draw (p2) -- (p22) ;

\filldraw (p1) circle (1pt);
\filldraw (p6) circle (1pt);
\filldraw (p17)  circle (1pt);
\filldraw (p20) circle (1pt);

\path (5,18) coordinate (p1); 
\path (6,21) coordinate (p2);  
\path (3,22) coordinate (p3); 
\path (3,20) coordinate (p4);  
\path (0,21) coordinate (p5);  
\path (1,18) coordinate (p6);  
\path (3,17) coordinate (p7);  

\path (9,16) coordinate (p17);  
\path (9,15) coordinate (nn);  
\path (8,19) coordinate (p18);   
\path (11,20) coordinate (p19);  
\path (12,17) coordinate (p20); 
\path (7,17) coordinate (p21); 
\path (9,21) coordinate (p22); 
\path (13,19) coordinate (p23); 
\path (10.8,15.7) coordinate (p24);
\path (11.5,13.5) coordinate (p25);

%\draw [blue] (10,10.5) node [above] {2};
\draw (p1) -- (p2) -- (p3)--(p1)--(p4)--(p3) -- (p5)--(p6)--(p3);

\draw [blue] (12,19.5) node [above] {4};

\draw [blue] (8.5,20) node [left] {2};

\draw [blue] (12.15,8) node [left] {3};

\draw [blue] (9.5, 19.5) node [above] {5};

\draw [blue] (11.5, 18.5) node [right] {1};

\draw [blue] (7.7,17.7) node [above] {2};

\draw [blue] (8.7,17.5) node [right] {1};

\draw [blue] (7.9,16.5) node [above] {3};

\draw [blue] (10.1,16.4) node [above] {2};

\draw [blue] (6,17.5) node [above] {5};

\draw [blue] (4,17) node [above] {2};

\draw [blue] (2,17) node [above] {2};

\draw [blue] (2,19) node [below] {1};

\draw [blue] (4,19) node [below] {2};

\draw [blue] (7.5, 21) node [above] {2} ;

\draw [blue] (3,20.5) node [right] {1};

\draw [blue] (3,18) node [above] {1};

\draw [blue] (5.5, 19.5) node [right] {1};

\draw [blue] (4, 20) node [right] {1};

\draw [blue] (1.5,21.5) node [above] {3};

\draw [blue] (4.5, 21.5) node [above] {1};

\draw [blue] (.5, 19.5) node [left] {4};

%\draw [red] (2, 10) node [left] {4};

\filldraw (p2) circle (1pt);
\filldraw (p3) circle (1pt);
\filldraw (p4) circle (1pt);

\filldraw (p5) circle (1pt);
\filldraw (p7) circle (1pt);

\filldraw (p18) circle (1pt);
\filldraw (p19)circle (1pt);
\filldraw (p21) circle (1pt);

\filldraw (p22) circle (1pt);

\filldraw (p23) circle (1pt);

\draw [->, very thick] (p6)-- (3,18) ;
\draw [ very thick] (3,18) -- (p1) ;
\draw [->, very thick] (p1) --(5.5,19.5) ;
\draw [very thick] (5.5,19.5) --  (p2)  ;
\draw [very thick] (4.5,21.5) -- (p2) ;
\draw [->, very thick] (p3) --(4.5, 21.5) ;
\draw[very thick]  (4,19)--(p1) ;
\draw [->, very thick] (p4)--(4,19);
\draw [->, very thick]  (p5) -- (1.5, 21.5) ;
\draw [very thick] (1.5, 21.5)-- (p3) ;
\draw [->, very thick] (p5) -- (.5, 19.5) ;
\draw [->, very thick] (p3) --( 3, 21) ;
\draw [very thick] (3,21) --(p4)--(2, 19)  ;
\draw [very thick] (.5, 19.5) -- (p6) ;
\draw [very thick]  (p1)--(4,20);
\draw [->, very thick]  (p3)--(4,20);

\draw [->, very thick] (p6)--(2,19) ;
\draw  [->, very thick] (p7)--(4, 17.5) ;
\draw [very thick] (4, 17.5)--(p1) ;

\draw [->, very thick] (p6) --(2, 17.5) ;
\draw [very thick] (2, 17.5) -- (p7) ;

\draw [->, very thick]  (p17)--(8.5, 17.5) ;
\draw [very thick] (8.5, 17.5)--(p18) ;

\draw [->, very thick] (p18) --(9.5,19.5)  ;
\draw [very thick] (9.5, 19.5)--(p19) ;

\draw [very thick] (p20)--(10.5, 16.5);
\draw [->, very thick] (p17)--(10.5, 16.5) ;

\draw [->, very thick] (p1)--(6,17.5) ;
\draw [very thick] (6, 17.5)--(7,17) ;

\draw [->, very thick] (7,17)--(8,16.5) ;

\draw [->, very thick] (7,17)--(7.5,18) ;
\draw [very thick] (7.5,18)--(p18) ;
\draw [very thick] (8,16.5)--(p17) ;

\draw [->, very thick] (p22)--(8.5,20) ;
\draw [very thick] (8.5, 20)--(p18) ;

\draw [ ->, very thick] (p20)--(12.5, 18) ;
\draw [very thick]  (12.5, 18) --(p23) ;
\draw [->, very thick]  (p19)-- (12, 19.5) ;
\draw [very thick] (12,19.5)--(p23) ;

\draw [->, very thick] (p19)--(11.5,18.5) ;
\draw [very thick] (11.5, 18.5)--(p20) ;

%\draw (p22)--(p21);
\draw [->, very thick] (p2) --(7.5, 21) ;
\draw [very thick ] (7.5,21)-- (p22) ;
\draw (p22) --(p19)  ;

\filldraw (p1) circle (1pt);
\filldraw (p6) circle (1pt);
\filldraw (p17)  circle (1pt);
\filldraw (p20) circle (1pt);

\draw (p5) node [left] {{}\hskip-5mm s};

\draw (p23) node [right] {t};

%\end{tikzpicture}
%\vskip-2mm\caption{Network and structure tree}\label{fig:Tree}\vskip-3mm
%\end{figure}

%\begin{figure}[htbp]
%\centering
%\begin{tikzpicture}[scale=.8]
%\draw (0,18) node { };
%\draw (0,0) node { };

\path (5,8) coordinate (p1); 
\path (6,11) coordinate (p2);  
\path (3,12) coordinate (p3); 
\path (3,10) coordinate (p4);  
\path (0,11) coordinate (p5);  
\path (1,8) coordinate (p6);  
\path (3,7) coordinate (p7);  

\path (9,6) coordinate (p17);  
\path (9,5) coordinate (nn);  
\path (8,9) coordinate (p18);   
\path (11,10) coordinate (p19);  
\path (12,7) coordinate (p20); 
\path (7,7) coordinate (p21); 
\path (9,11) coordinate (p22); 
\path (13,9) coordinate (p23); 
\path (10.8,5.7) coordinate (p24);
\path (11.5,3.5) coordinate (p25);

\draw [dashed, red]  (4,13) --(7,8)--(12, 6);
%\draw [blue] (10,10.5) node [above] {2};

\draw [blue] (12,9.5) node [above] {4};

\draw [blue] (8.5,10) node [left] {2};

\draw [blue] (12.5,8) node [left] {3};

\draw [blue] (9.5, 9.5) node [above] {5};

\draw [blue] (11.5, 8.5) node [right] {1};

\draw [red] (7.7,7.7) node [above] {2};

\draw [red] (8.7,7.5) node [right] {1};

\draw [blue] (7.9,6.5) node [above] {3};

\draw [red] (10.1,6.4) node [above] {2};

\draw [blue] (6,7.5) node [above] {5};

\draw [blue] (4,7) node [above] {2};

\draw [blue] (2,7) node [above] {2};

\draw [blue] (2,9) node [below] {1};

\draw [blue] (4,9) node [below] {2};

\draw [blue] (7.5, 11) node [above] {2} ;

\draw [blue] (3,10.5) node [right] {1};

\draw [blue] (3,8) node [above] {1};

\draw [red] (5.5, 9.5) node [right] {1};

\draw [blue] (4, 10) node [right] {1};

\draw [blue] (1.5,11.5) node [above] {3};

\draw [red] (4.5, 11.5) node [above] {1};

\draw [blue] (.5, 9.5) node [left] {4};

%\draw [red] (2, 10) node [left] {4};

\filldraw (p2) circle (1pt);
\filldraw [red] (p3) circle (2pt);
\filldraw [red] (p4) circle (2pt);

\filldraw [red]  (p5) circle (2pt);
\filldraw [red] (p7) circle (2pt);

\filldraw (p18) circle (1pt);
\filldraw (p19)circle (1pt);
\filldraw [red] (p21) circle (2pt);

\filldraw (p22) circle (1pt);

\filldraw (p23) circle (1pt);

\draw (p1)--(p6) ;
\draw (p1) -- (p2) -- (p3) -- (p1) ;
\draw (p1)--(p4)--(p3) -- (p5)--(p6)--(p3);

\draw (p4)--(p6) ;

\draw (p7) -- (p6) ;

\draw (p17)--(p18)--(p19)--(p20)--(p17);
\draw (p1)--(p17) ;
\draw (p20)--(p23)--(p22)--(p22)--(p21);
\draw (p2) -- (p22) ;
\draw [red, very thick]   (p4)--(p1)-- (p21) -- (p17) ;
\draw [red, very thick]  (p4)-- (p3) --(p6) -- (p5)  ;
\draw [red, very thick]  (p1)--(p7)  ;
\filldraw [red] (p1) circle (2pt);
\filldraw [red] (p6) circle (2pt);
\filldraw [red] (p17)  circle (2pt);
\filldraw (p20) circle (1pt);

\draw (p5) node [left] {{}\hskip-5mm s};

\draw (p23) node [right] {t};

\end{tikzpicture}
\vskip-2mm\caption{Max-Flow Min-Cut Theorem}\label{fig:Tutte}\vskip-3mm \end{figure}
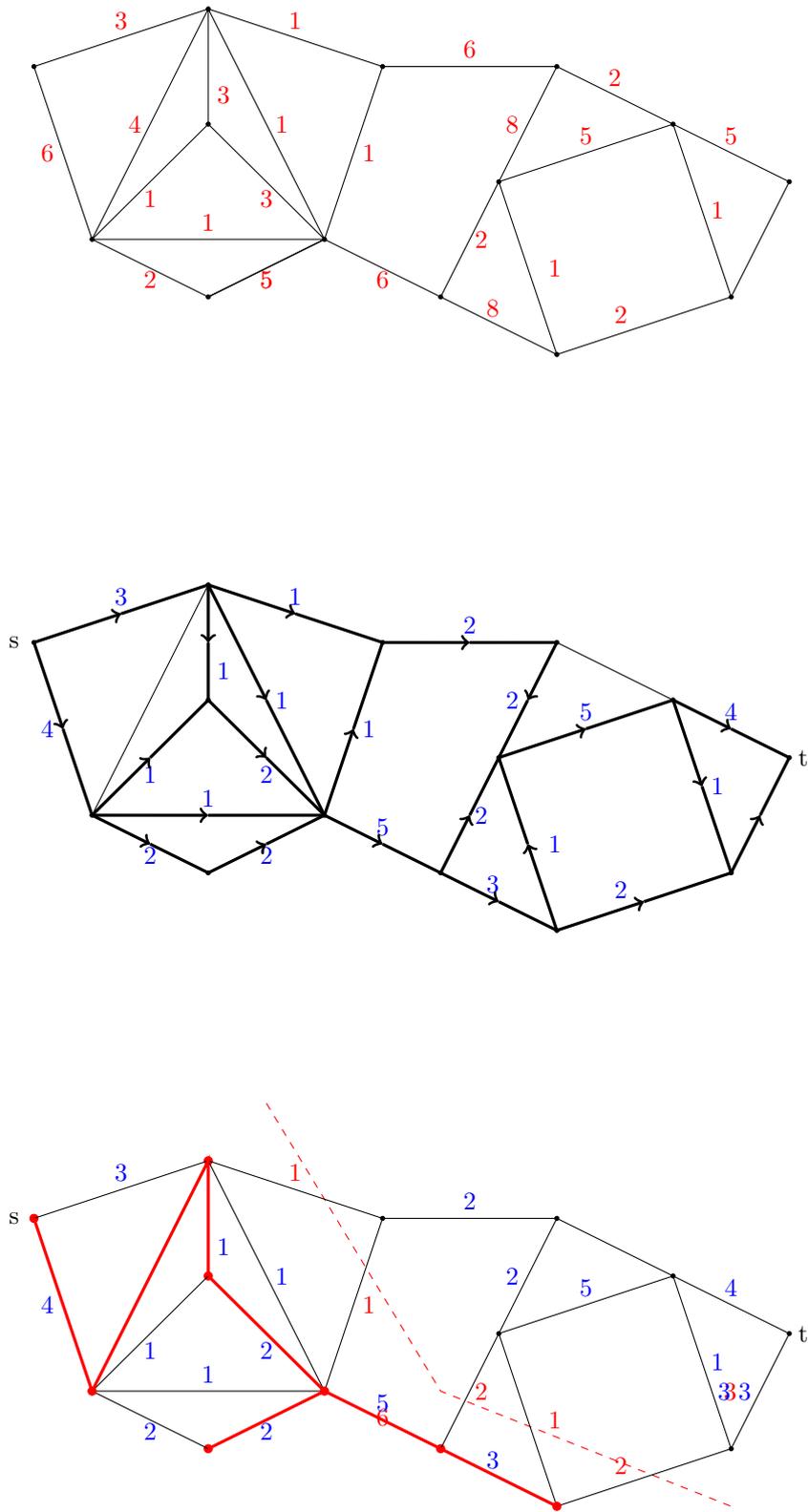
In the proof of this result it is shown that one obtains a min-cut from a max-flow as the set of vertices that are connected to $s$ by a path in which each edge has some unused capacity.    Thus in Figure \ref {fig:Tutte}  the  min-cut  vertices are shown in red and the  edges with unused capacity used in the construction of the max-flow are also shown
in red.

In this paper it is shown that for any network there is a uniquely determined network based on a {\it structure tree} that provides a convenient way of encoding the minimal flow between
any pair of vertices.  Specifically we will prove the following theorem.

\begin {theo}\label {maintheorem}  Let $N(X)$ be a network.    There is a uniquely determined network $N(T)$ based on a tree  $T$ and an injective map $\nu : VX \rightarrow VT$, such
that  the maximum value of an $(s,t)$-flow in $X$ is the maximum value of a $(\nu s, \nu t)$-flow in $N(T)$.   Also,  for any edge $e' \in ET$, there are vertices
$s, t \in VX$ such that $e'$ is on the geodesic joining $\nu s$ and $\nu t$ and $c(e')$ is the capacity of a minimal $(s,t)$-cut.

\end {theo}

An example of a network and its structure tree are shown in Figure \ref {fig:Tree} .      Thus in this network the max-flow between $u$ and $p$ is $12$.
One can read off a corresponding min-cut by removing the corresponding edge from the  structure tree.  Thus a min-cut separating $u$ and $p$ is
 $\{q,  r, s, t, u, v, w  \} $.    The map $\nu $ need not be surjective.    In our example there is a single vertex  $z$  that is not in the image
 of $\nu $ shown in bold.     One can get a structure tree for which $\nu $ is bijective by contracting one of the four edges incident with this vertex.
 The tree then obtained is a Gomory-Hu tree \cite {GH}.
The structure tree constructed in the proof of Theorem \ref {maintheorem} is uniquely determined and is therefore invariant under the automrophism
group of the network.    The tree obtained by contracting one of the four edges is no longer uniquely determined as one gets a different
tree for each of the four choices.
 In some cases this would mean that the structure tree  did not admit
 the automorphism group of the network.   Thus for example if the automorphism group of $X$ is transitive on $VX$ and $c(e) = 1$ for every edge,
 then the structure tree would have $n$ vertices of degree one, where $n = |VX|$ and one vertex of degree $n$.  Clearly this structure tree will admit
 the automorphism group of $X$, but if one edge is contracted to get a tree with $n$ vertices, then the new tree will not admit the automorphism group.
 
 Not every min-cut separating a pair of vertices can be obtained from the structure tree.   The min-cuts obtained are the ones that are optimally nested with the cuts of equal or smaller capacity.    In our example there are four cuts of capacity $12$ corresponding to edges in the structure tree incident with $z$.   However there are other cuts of capacity $12$.    Thus there are two min-cuts in the structure tree separating $k$ and $h$,
 but there are in fact four min-cuts separating $k$ and $h$.    In  \cite {[DKL]}  it is shown that the min-cuts separating two vertices correspond to the edge cuts 
 in a cactus, which is a connected graph in which each edge belongs to at most one cycle.
 The cactus of min-cuts separating $k$ and $h$ is a $4$-cycle.

\begin{figure}[htbp]
\centering
\begin{tikzpicture}[scale=.7]
%\draw (0,25) node { } ;
\
%\path (3,8.7) coordinate (a1); 
%\path (3,7.3) coordinate (a2); 
\path (5,8) coordinate (p1); 
\path (6,11) coordinate (p2);  
\path (3,12) coordinate (p3); 
\path (3,10) coordinate (p4);  
\path (0,11) coordinate (p5);  
\path (1,8) coordinate (p6);  
\path (3,7) coordinate (p7);  
\path (3,4) coordinate (p8);  
\path (1,4) coordinate (p9);  
\path (1,2) coordinate (p10);  
\path (3,2) coordinate (p11);  
\path (5,4) coordinate (p12);   
\path (5,2) coordinate (p13);   
\path (5,0) coordinate (p14);   
\path (7,2) coordinate (p15);   
\path (7,4) coordinate (p16);   
\path (9,6) coordinate (p17);  
\path (9,5) coordinate (nn);  
\path (8,9) coordinate (p18);   
\path (11,10) coordinate (p19);  
\path (12,7) coordinate (p20); 
\path (7,7) coordinate (p21); 
\path (9,11) coordinate (p22); 
\path (13,9) coordinate (p23); 
\path (10.8,5.7) coordinate (p24);
\path (11.5,3.5) coordinate (p25);

\draw [red] (10,10.5) node [above] {2};

\draw [red] (12,9.5) node [above] {1};

\draw [red] (8.5,10) node [left] {8};

\draw [red] (12.5,8) node [left] {3};

\draw [red] (9.5, 9.5) node [above] {2};

\draw [red] (11.5, 8.5) node [right] {1};

\draw [red] (7.7,7.7) node [above] {2};

\draw [red] (8.7,7.5) node [right] {1};

\draw [red] (7.9,6.5) node [above] {8};

\draw [red] (10.1,6.4) node [above] {3};

\draw [red] (7.5,4.5 ) node [above] {2};

\draw [red] (8,4) node [right] {1};

\draw [red] (6,1) node [above] {8};

\draw [red] (7,3) node [right] {3};

\draw [red] (7,5 ) node [above] {3};

\draw [red] (6,4) node [above] {1};

\draw [red] (6,2) node [above] {4};

\draw [red] (5,3) node [right] {7};

\draw [red] (5,1) node [right] {1};

\draw [red] (6,7.5) node [above] {6};

\draw [red] (4,3) node [above] {1};

\draw [red] (4,2) node [above] {4};

\draw [red] (4,1) node [above] {7};

\draw [red] (4,7) node [above] {5};

\draw [red] (2,3) node [right] {5};

\draw [red] (2,2) node [above] {1};

%\draw [red] (3,5) node [right] {4};

\draw [red] (2,7) node [above] {2};

\draw [red] (4,7) node [above] {5};

%\draw [red] (2, 6) node [left] {1};

%\draw [red] (4,6) node [right] {4};

\draw [red] (2,9) node [below] {1};

\draw [red] (4,9) node [below] {2};

\draw [red] (1, 6) node [left] {6};

\draw [red] (1,3) node [left] {4};

\draw [red] (3,10.5) node [right] {3};

\draw [red] (3,8) node [above] {1};

\draw [red] (5.5, 9.5) node [right] {6};

\draw [red] (4, 10) node [right] {1};

\draw [red] (1.5,11.5) node [above] {3};

\draw [red] (4.5, 11.5) node [above] {1};

\draw [red] (.5, 9.5) node [left] {6};

\draw [red] (2, 10) node [left] {4};

%\filldraw[rounded corners=4pt,black!15]  (2.6,5)--(2.6,6.5)--(3.4,6.5)--(3.4,3.5)--(2.6,3.5)--(2.6,5);
%\filldraw[rounded corners=4pt,black!15]  (4.6,2)--(4.6,4.4)--(7.4,4.4)--(7.4,1.7)--(5.3,-0.4)--(4.6,-0.4)--(4.6,2);
%\filldraw (a1) [black!15] circle (12pt);
%\filldraw (a1) circle (3pt);
%\filldraw (a2) [black!15] circle (12pt);
%\filldraw (a2) circle (3pt);
%\filldraw (p2) [black!15] circle (12pt);
\filldraw (p2) circle (1pt);
\filldraw (p3) circle (1pt);
\filldraw (p4) circle (1pt);
%\filldraw (p5) [black!15] circle (12pt);
\filldraw (p5) circle (1pt);
\filldraw (p7) circle (1pt);
%\filldraw (p8) circle (1pt);
\filldraw (p9) circle (1pt);
%\filldraw (p10) [black!15] circle (12pt);
\filldraw (p10) circle (1pt);
\filldraw (p11) circle (1pt);
\filldraw (p12) circle (1pt);
\filldraw (p13) circle (1pt);
\filldraw (p14) circle (1pt);
\filldraw (p15) circle (1pt);
\filldraw (p16) circle (1pt);
\filldraw (p18) circle (1pt);
\filldraw (p19)circle (1pt);
\filldraw (p21) circle (1pt);
%\filldraw (p22) [black!15] circle (12pt);
\filldraw (p22) circle (1pt);
%\filldraw (p23) [black!15] circle (12pt);
\filldraw (p23) circle (1pt);
%\filldraw (p24) [black!15] circle (12pt);
%\filldraw (p24) circle (1pt);
%\filldraw (p25) [black!15] circle (12pt);
%\filldraw (p25) circle (1pt);

\draw (p1)--(p6) ;
\draw (p1) -- (p2) -- (p3)--(p1)--(p4)--(p3) -- (p5)--(p6)--(p3);
\draw (p4)--(p6) ;
\draw  (p1)--(p7)--(p1) ;
\draw (p7) -- (p6) -- (p9)--(p10)--(p11)--(p9);
\draw (p11)--(p15) ;-
\draw (p12)--(p14)--(p11)--(p12);
\draw (p14)--(p15)--(p16)--(p12)--(p17)--(p16);
\draw (p15)--(p17)--(p18)--(p19)--(p20)--(p17);
\draw (p1)--(p17) ;
\draw (p20)--(p23)--(p22)--(p22)--(p21);
%\draw (p17)--(p25)--(p20);

\filldraw (p1) circle (1pt);
\filldraw (p6) circle (1pt);
\filldraw (p17)  circle (1pt);
\filldraw (p20) circle (1pt);

\draw (p1) node [right] {\hskip1.5mm a};
\draw (p2) node [right] {\hskip0.6mm b};
\draw (p3) node [above] {c};
\draw (p4) node [below] { d};
%\draw (a1) node [above] {5};
%\draw (a2) node [below] {6};
\draw (p7) node [below] {\hskip2.5mm e};
%\draw (p8) node [below] {f};
\draw (p5) node [left] {{}\hskip-5mm g};
\draw (p6) node [left] {h};
\draw (p9) node [left] {i};
\draw (p10) node [left] { j};
\draw (p11) node [below] {\hskip-2mm k};
\draw (p12) node [above] {\hskip-2mm l};
\draw (p13) node [above] {\hskip-5mm m};
\draw (p14) node [below] {n};
\draw (p16) node [below] {\hskip-4.4mm o};
\draw (p15) node [below] {\hskip2mm p};
\draw (nn) node [above] {\hskip0.8mm q};
\draw (p18) node [left] {r};
\draw (p19) node [above] {\hskip2mm s};
\draw (p20) node [right] {t};
\draw (p21) node [below] {\hskip-1mm u};
\draw (p22) node [above] {v};
\draw (p23) node [right] {w};
%\draw (p24) node [below] {\hskip1mm 23};
%\draw (p25) node [below] {27};
%\draw (7.7,-1.5) node {13a};

\end{tikzpicture}
%\vskip-2mm\caption{Network and structure tree}\label{fig:Tree}\vskip-3mm
\end{figure}

\begin{figure}[htbp]
\centering
\begin{tikzpicture}[scale=.7]
%\draw (0,25) node { } ;
\

%\path (3,8.7) coordinate (a1); 
%\path (3,7.3) coordinate (a2); 
\path (5,8) coordinate (p1); 
\path (6,11) coordinate (p2);  
\path (3,12) coordinate (p3); 
\path (3,10) coordinate (p4);  
\path (0,11) coordinate (p5);  
\path (1,8) coordinate (p6);  
\path (3,7) coordinate (p7);  
\path (4,5) coordinate (p8);  
\path (1,4) coordinate (p9);  
\path (1,2) coordinate (p10);  
\path (3,2) coordinate (p11);  
\path (5,4) coordinate (p12);   
\path (5,3) coordinate (p13);   
\path (5,0) coordinate (p14);   
\path (6.5,2) coordinate (p15);   
\path (6.5,3) coordinate (p16);   
\path (9,6) coordinate (p17);  
\path (9,5) coordinate (nn);  
\path (8,9) coordinate (p18);   
\path (11,10) coordinate (p19);  
\path (12,7) coordinate (p20); 
\path (7,7) coordinate (p21); 
\path (9,11) coordinate (p22); 
\path (13,9) coordinate (p23); 
\path (10.8,5.7) coordinate (p24);
\path (11.5,3.5) coordinate (p25);

%\draw [red] (10,10.5) node [above] {2};

%\draw [red] (12,9.5) node [above] {1};

\draw [red] (8.5,10) node [left] {10};

\draw [red] (12.5,8) node [left] {4};

\draw [red] (9.5, 9.5) node [above] {6};

%\draw [red] (11.5, 8.5) node [right] {1};

%\draw [red] (7.7,7.7) node [above] {2};

\draw [red] (8.5,7.5) node [right] {5};

\draw [red] (7.9,6.5) node [above] {16};

\draw [red] (10.1,6.4) node [above] {5};

\draw [red] (6.8,5.5 ) node [above] {12};

%\draw [red] (8,4) node [right] {1};

%\draw [red] (6,1) node [above] {8};

%\draw [red] (7,3) node [right] {3};

\draw [red] (4,3.8 ) node {12};

%\draw [red] (6,4) node [above] {1};

\draw [red] (5,2) node [below] {16};

\draw [red] (5,3.5) node [right] {12};

\draw [red] (4,1) node [right] {16};

%\draw [red] (6,7.5) node [above] {2};

\draw [red] (5.5,2.5) node [left] {6};

\draw [red] (4,2.5) node [above] {14};

%\draw [red] (4,1) node [above] {7};

\draw [red] (4,7) node [above] {11};

\draw [red] (2.4,4.2) node {12};

%\draw [red] (4,6) node [right] {9};

\draw [red] (2,9) node [below] {6};

%\draw [red] (4,9) node [below] {2};

\draw [red] (3, 6) node [left] {12};

\draw [red] (1,3) node [left] {5};

%\draw [red] (3,10.5) node [right] {3};

\draw [red] (3,8) node [above] {10};

\draw [red] (5.5, 9.5) node [right] {7};

%\draw [red] (4, 10) node [right] {1};

%\draw [red] (1.5,11.5) node [above] {3};

%\draw [red] (4.5, 11.5) node [above] {1};

\draw [red] (.5, 9.5) node [left] {9};

\draw [red] (2, 10) node [left] {12};

%\filldraw[rounded corners=4pt,black!15]  (2.6,5)--(2.6,6.5)--(3.4,6.5)--(3.4,3.5)--(2.6,3.5)--(2.6,5);
%\filldraw[rounded corners=4pt,black!15]  (4.6,2)--(4.6,4.4)--(7.4,4.4)--(7.4,1.7)--(5.3,-0.4)--(4.6,-0.4)--(4.6,2);
%\filldraw (a1) [black!15] circle (12pt);
%\filldraw (a1) circle (3pt);
%\filldraw (a2) [black!15] circle (12pt);
%\filldraw (a2) circle (3pt);
%\filldraw (p2) [black!15] circle (12pt);
\filldraw (p2) circle (1pt);
\filldraw (p3) circle (1pt);
\filldraw (p4) circle (1pt);
%\filldraw (p5) [black!15] circle (12pt);
\filldraw (p5) circle (1pt);
\filldraw (p7) circle (1pt);
\filldraw (p8) circle (3pt);
\filldraw (p9) circle (1pt);
%\filldraw (p10) [black!15] circle (12pt);
\filldraw (p10) circle (1pt);
\filldraw (p11) circle (1pt);
\filldraw (p12) circle (1pt);
\filldraw (p13) circle (1pt);
\filldraw (p14) circle (1pt);
\filldraw (p15) circle (1pt);
\filldraw (p16) circle (1pt);
\filldraw (p18) circle (1pt);
\filldraw (p19)circle (1pt);
\filldraw (p21) circle (1pt);
%\filldraw (p22) [black!15] circle (12pt);
\filldraw (p22) circle (1pt);
%\filldraw (p23) [black!15] circle (12pt);
\filldraw (p23) circle (1pt);
%\filldraw (p24) [black!15] circle (12pt);
%\filldraw (p24) circle (1pt);
%\filldraw (p25) [black!15] circle (12pt);
%\filldraw (p25) circle (1pt);

\draw (p1)--(p6);
\draw (p1) -- (p2) ;
%\draw (p3)--(p1)--(p4)--(p3) ;
\draw (p5)--(p6)--(p3);
\draw (p4)--(p6) ;
\draw  (p1)--(p7);
\draw (p6)--(p8)--(p9)--(p10)  ;
\draw (p11)--(p13);
\draw (p11)--(p15) ;-
\draw (p12)--(p13);
\draw (p17)--(p8)--(p11) ;
\draw (p14)--(p11) ;
\draw (p17)--(p18)--(p19) ;
\draw (p20)--(p17);
\draw (p21)--(p17) ;
\draw (p20)--(p23) ;
\draw (p22)--(p18) ;
\draw (p16)--(p11);

\filldraw (p1) circle (1pt);
\filldraw (p6) circle (1pt);
\filldraw (p17)  circle (1pt);
\filldraw (p20) circle (1pt);

\draw (p1) node [right] {\hskip1.5mm a};
\draw (p2) node [right] {\hskip0.6mm b};
\draw (p3) node [above] {c};
\draw (p4) node [below] { d};
%\draw (a1) node [above] {5};
%\draw (a2) node [below] {6};
\draw (p7) node [below] {\hskip2.5mm e};
\draw (p8) node [above] {\hskip4mm z};
\draw (p5) node [left] {{}\hskip-5mm g};
\draw (p6) node [left] {h};
\draw (p9) node [left] {i};
\draw (p10) node [left] { j};
\draw (p11) node [below] {\hskip-2mm k};
\draw (p12) node [above] {\hskip-2mm l};
\draw (p13) node [above] {\hskip-5mm m};
\draw (p14) node [below] {n};
\draw (p16) node [left,above] { o};
\draw (p15) node [below] {\hskip2mm p};
\draw (p17) node [below] {\hskip0.8mm q};
\draw (p18) node [left] {r};
\draw (p19) node [above] {\hskip2mm s};
\draw (p20) node [right] {t};
\draw (p21) node [below] {\hskip-1mm u};
\draw (p22) node [above] {v};
\draw (p23) node [right] {w};
%\draw (p24) node [below] {\hskip1mm 23};
%\draw (p25) node [below] {27};
%\draw (7.7,-1.5) node {13a};

\end{tikzpicture}
\vskip-2mm\caption{Network and structure tree}\label{fig:Tree}\vskip-3mm
\end{figure}
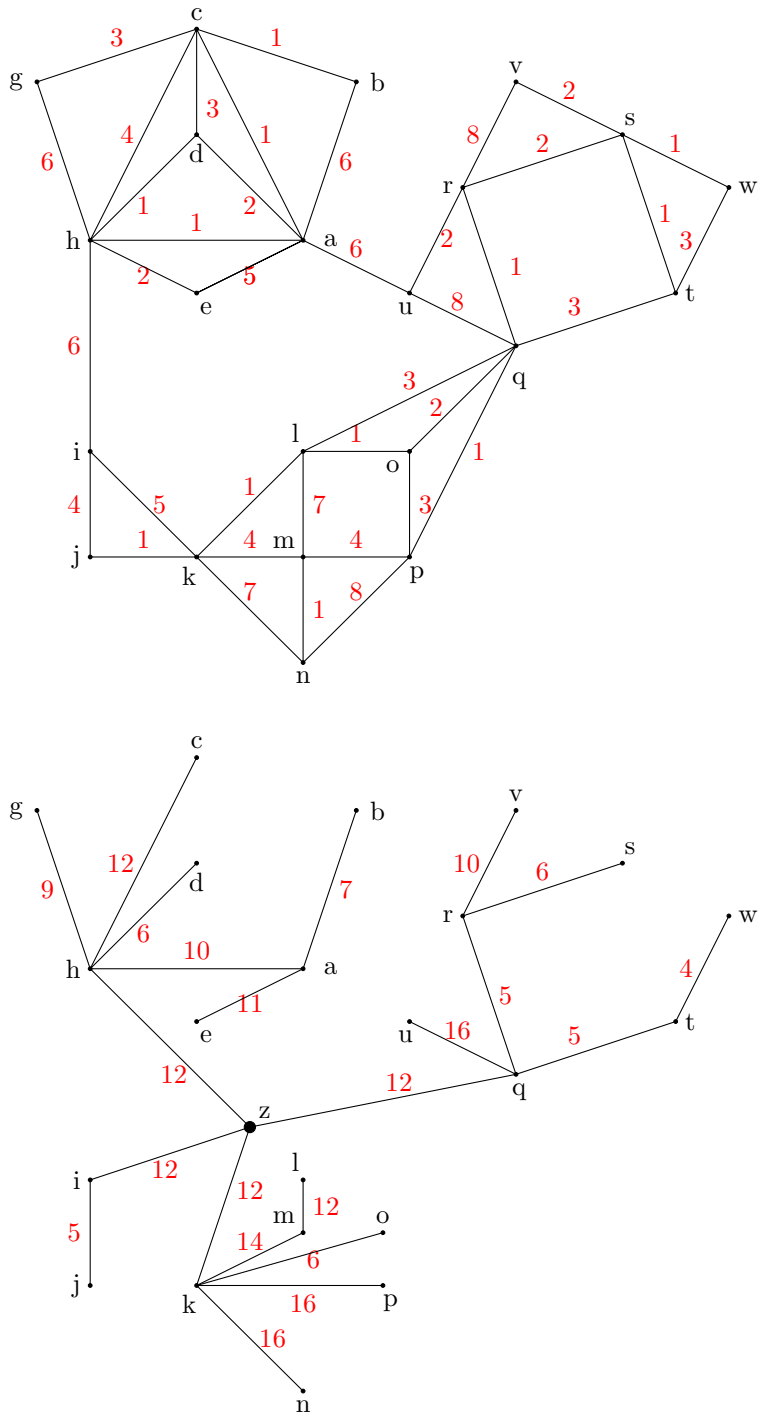

%\section {Proof of Theorem \ref {maintheorem}}
\begin {proof}[Proof of Theorem \ref {maintheorem}]
Let $N$ be a network based on the graph $X$.
Let $\B X$ denote the set of all cuts in $X$.   It is not hard to show that $\B X$ is a Boolean ring invariant  under the automorphism group
of $X$.  Let $\B _n$ be the subring generated by those $A \in \B X$ for which $c(A)   \leq n$.

A cut $A$ is defined to be {\it $n$-thin}   if  $c(A) = n$  but $A \notin \B _{n-1}$.    We say $A$ is thin if it is $n$-thin for some $n$.
Alternatively a   cut  $A$  is  defined to be  {\it thin } (or thin  with respect to $u,v \in VX$) if it separates some $u,v  \in VX$ and $c(A)$ is minimal 
 among all the cuts that separate $u$ and $v $.  These two definitions coincide.
 
 \begin {prop} \label {thin} A  cut $A$ is such that $c(A) =n$ and $A \notin \B_{n-1}$  if and only if there are vertices $u, v \in VX$  with respect to which $A$ is thin.
 \end {prop}
 \begin {proof}   
 
 Suppose $c(A) = n$ and $A\notin \B _{n-1}$.    Let $u \in A , v \in A^*$, and suppose  there exists a cut $B(u,v)$ such that 
 $c(B(u,v)) < n$ and $u \in B(u,v), v \in B(u,v)^*$.    If $B_u   = \bigcap (B(u,v) | v \in A^*) $, then $A = \bigcup ( B_u | u \in A)$, and so $A\in \B _{n-1}$, which is a contradiction. 
 Thus there are vertices $u,v$ such that $A$ separates $u$ and $v$ but they are not separated by a $B$ with $c(B) < n$.
 Conversely suppose $A$ separates $u, v$.    If $A \in \B _{n-1}$,  then $A$ is a can be written as a finite union of intersections of cuts 
 $B_1, B_2, \dots , B_k$ with $c(B_i) < n, i = 1, 2, \dots k$.   If no $B_i$ separates $u, v$ then neither will $A$.  If $A$ separates $u,v$
 then $u,v$ are separated by some $B_i$.    If  $A$ is $n$-thin with respect to  $u,v$, then $A \notin \B_{n-1}$.
 \end{proof}
 
  Let $\C _n$ be the set of thin cuts $A$ with $c(A)  \leq n$.

If $A, B$ are cuts, then the sets  $A\cap B, A^*\cap B, A^*\cap B, A^*\cap B, A\cap B^*$ are also cuts.  These sets are called the $\it corners $ of
$A, B$.   This term is suggested by Figure \ref {Cuts} .   Two corners are called opposite or adjacent as suggested in  this figure.
We say two cuts  $A, B$ are {\it nested} if  one $A\cap B, A^*\cap B, A^*\cap B, A^*\cap B, A\cap B^*$ is empty.
  A set $\ce $ of cuts is said to be nested if any two elements of $\ce $ are nested.
We consider sets  $\ce $ satisfying the following conditions:-
\begin{itemize}
\item [(i)] If $A \in \ce$, then $A^* \in \ce$.
\item [(ii)]  The set $\ce $ is nested.
%\item [(iii)] If $A, B \in \ce$ and $A \subset B$, then there are only finitely  many $C \in \ce$ such that $A\subset C\subset B$.
\end {itemize}
\begin{theo} If $\ce $ is a set satisfying conditions (i)  and (ii), then there is a tree $T(\ce)$  such that the directed edge set  is $ \ce $. 
\end{theo}
\begin {proof}  We define $VT$ to be the set  of maps $\alpha  : \ce \rightarrow \Z _2$ satisfying the following
\begin{itemize}
\item [(a)]  If $\alpha (A) = 1$, then $\alpha (A^*) = 0.$
\item [(b)]  If  $\alpha (A) =1$ and $A\subset B$, then $\alpha  (B) =1$.

\end{itemize}
Put $ET = \ce $ and for  $A \in \ce$, put $\iota A =  \alpha  $ where $\alpha  (B) = 1 $ if $A\subseteq B$ or if  $A^* \subset B$.   Put $\tau A = \iota A^*$.
Then $\iota A $ and $\tau A$ take the same value on every $B$ except if $B = A$ or $B =A^*$.
It is fairly easy to check that $\iota A$ satisfies conditions  (a) and (b).    If $u =\iota A$ and $v = \tau B$, then the directed edges in a path joining
$u$ and $v$ consist of the set $\{ C \in \ce | A \subseteq C \subseteq B\}$.  This set is totally ordered by inclusion and so is the unique geodesic joining $u$ and $v$.  Thus $T$ is a tree.

We can identify  a vertex $v$  of $X$ with a map $v : \B X \rightarrow  \Z _2$.  Thus $v(A) = 1$ if $v \in A$ and $v(A) = 0$ if $v \notin A$.
  Restricting to $\ce $ will give a vertex of $T$.  Thus there is a map $\nu :  VX \rightarrow VT$ such that $\nu (\alpha ) , \nu (\beta )$ differ only on the cuts separating $\alpha $ and $\beta $.

\end{proof}
  
  Note that there may be vertices of $T$ which are not in the image
of $\nu $.   

If $\ce \subset \B X$ satisfies the above conditions, then there is a tree $T(\ce )$.   If $G$ is the automorphism group of $X$ and $\ce $ is a $G$-set, then $T$ is called
a {\it structure tree} for $X$.    If $T = T(\ce )$ is a structure tree for $X$,  then there is a $G$-map $\nu : VX \rightarrow VT$

Our proof of  Theorem \ref{maintheorem} is  by finite induction.   We show that there is  a nested $G$-set $\ce _n$  of thin cuts that generates $\B _n$,
and $\ce _{n-1} \subseteq \ce _{n}$ for each $n$.

\begin {lemma}\label {corner}   Let  $A, B, C$  be cuts .    
\begin {itemize}
\item [(i)] Let $A, B$ be not nested  and let $C$ be nested with both $A$ and $B$, then $C$ is nested with every
corner of $A, B$.  
\item [(ii)]If $C$ is nested with $A$, then $C$ is nested with two adjacent corners of $A$ and $B$.
\end {itemize}
\end {lemma}

\begin{proof} 
For (i) by possibly  relabelling   $A$ as $A^*$ and/or $B$ as $B^*$ and/or $C$ as $C^*$  we can assume either
\begin{itemize}
\item [(a)]  $C \subset A$ and $C\subset B$
or 
\item[(b)] $C\subset A$ and $C^*\subset B$.    
\end{itemize}
If (a) then $C\subset A\cap B$ and $C$ is contained in the complement of each of the other corners.
If (b), then   $B^*\subset C\subset A$, and so $A, B$ are nested, which contradicts our hypothesis.

For (ii) if $A \subset C$, then $A\cap B \subset C$ and $A\cap B^* \subset C$.
\end{proof}

Let $\C $ be a set of  cuts 
Let $A$ be a cut and let $M(A, \C)$ be the set of  cuts in $\C $  which are not  nested with $A$. Set $\mu  (A,\C) = |M(A, \C)|$.

\begin{lemma}\label{corners_equality}
Let $\C $ be a nested set of  thin cuts.   Let $B \in \C$ and let $A$ be a thin cut which is not nested with some $B\in \C$, then

\[\mu(A\cap B, \C ) + \mu (A\cap B^*, \C)  <  \mu (A,\C) .\  \ \ \  \]
\end{lemma}

\begin{proof}
  If $C \in \C$ is nested with $A$, then it is nested with both $A$ and $B$ and so it is nested with $A\cap B$ and $A\cap B^*$ by Lemma \ref{corner}.
If $C$ is not nested with $A$, then it  must be nested with one of $A\cap B$ and $A\cap B^*$.   For if, say, $C \subset B$ then $B^* \subset C^*$ and so $A\cap B^* \subset C^*$.
Thus $C$ is not nested with at most one of  $A\cap B$ and $A\cap B^*$  and the lemma follows, since $B$ is counted on the right but not on the left.
\end{proof}

If $\ce _{n-1}$ does not generate $\B _nX$, then there is a thin cut $A \in \C_n - \B _{n-1}$.
We will show that $\B _n$ is generated by a set $\ce _{n-1} \cup \C _n'$, where $\C _n'$ is the set of cuts $A \in \C _n - \B_{n-1}$ that are nested with
every $C \in \ce _{n-1}$.    

To see this, let $A \in \C _n - \B_{n-1}$.   If $A$ is not nested with some $B \in \ce _{n-1}$, then all four corners of $A, B$ are not empty.
We refer to Figure \ref{Cuts}  .    
 By relabelling $A^*$ as $A$ and $B^*$ as $B$ if necessary 
we can assume $a\leq b, c\leq d$.  
    Suppose $a < c$.   Then     $B = B\cap A + B\cap A^*$  and $c(B\cap A) = a +c +f < c +d +e +f =c(B), c(B\cap A^*) = a +e +d < c +d +e +f =c(B)$, which contradicts the fact that $B$ is thin.  Thus $a \geq c$.   If $a > c$, then a similar argument shows 
$A$ is not thin.   Thus $a = c$.     
Also $A = A\cap B + A\cap B^*$ and both $c(A\cap B) \leq c( A)$ and $c(A\cap B^*) \leq c( A)$, so that either one of the corners
$A\cap B, A\cap B^*$ is in $\C_n $ and the other in $\B _{n-1}$ or both corners are in $\C _n$.
Since $B $ and $A$ are not nested we have   $\mu (A, \ce _{n-1}) \not= 0$, and  by Lemma \ref {corners_equality}   we have $\mu (A\cap B, \ce _{n-1})+\mu (A\cap B^*, \ce _{n-1}) < \mu (A, \ce _{n-1})$.  Thus $A = A\cap B + A\cap B^*$ and both $A\cap B$ and $A\cap B^*$ are not nested with fewer cuts in $\ce _{n-1}$  than $A$. 
 An easy induction argument now shows that $\B _n$ is generated by the set of cuts  $\ce _{n-1} \cup \C _n'$, where $\C _n'$ is the set of cuts $A \in C_n - \B_{n-1}$ that are nested with
every $C \in \ce _{n-1}$.

We now show that we can restrict $\C _n'$ further so that it becomes a nested set.

Let $u \in VX$ be a vertex that is separated from a vertex $v \in VX$ by some $A \in \C _n'$ but is not separated from $v$ by any $A \in \ce _{n-1}$.
We show that there is a smallest $A _u \in \C _n'$ that contains $u$.   

Let $A, B \in \C _n'$ and let $u \in A\cap B$.  We show that $A\cap B \in \C _n'$.      Clearly $A\cap B \in \C _n'$ if $A, B$ are nested.
If $A, B$ are not nested, then $A\cap B$ is nested with every $C \in \ce _{n-1}$ by Lemma \ref {corner}.

Again we refer to Figure \ref{Cuts}.    We see that $c(A\cap B) + c(A^*\cap B^*) \leq c(A) + c(B)$ with equality if and only if $e = 0$.
and that $c(A^*\cap B) + c(A\cap B^*) \leq c(A)+c(B)$ also with equality if and only if $f = 0$.    It is not possible that two adjacent 
corners of $A, B$ are in $\B _{n-1}$.  Thus for one pair of opposite corners we have that both corners are in $\C _n'$.
It those two corners are $A\cap B$ and $A^*\cap B^*$ they we are done.   If not, then $c(A^*\cap B) = c(A\cap B^*)  =n$,  and $f = 0$.  
 But since $A\cap B$ separates $u$ from $v$ and $c(A\cap B) \leq n$  it must be in $\C _n'$ with  $e = 0$ and $c(A\cap B) = n$
It follows easily that there is a smallest $A_u \in \C _n'$ containing $u$.       

Let $w \in VX$ be such that $w$ is separated from a $z \in VX$ by some $A\in C_n '$.  It is easy to show that $A_u$ and $A_w$ are equal
or disjoint.  If we take the set of all such $A_u$ then these sets together with $\ce _{n-1}$ form a nested set of generators for $\B _n$.

This concludes the proof of Theorem \ref {maintheorem}
\end {proof}
Essentially the same proof gives the following  stronger result.
\begin {theo} Let $N$ be a network based on the graph $X$ , and let $\B X$ be the Boolean ring defined above.   Let $\cs $ be a subring of $\B X$.
Then $\cs $ has a nested set of generators $\ce = \ce (\cs)$.
such that there is a network $N(T)$ based on a tree  $T = T(\ce )$ and an injective map $\nu : VX \rightarrow VT$, such
that  if two vertices $u,v$ are separated by a cut $A$ in $\cs $ then $\nu u, \nu v$  are separated by a cut $C \in \ce $ with $c(C) \leq c(A)$. Also,  for any edge $e' \in ET$, there are vertices
$s, t \in VX$ such that $e'$ is on the geodesic joining $\nu s$ and $\nu t$ and   $c(e')$ is the capacity of a minimal $(s,t)$-cut in $\cs $.
\end {theo}

%\eject
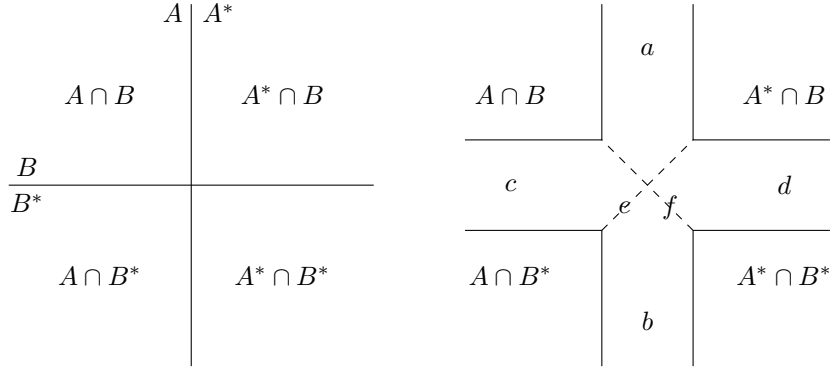
\begin{figure}[ht!]

\centering
\begin{tikzpicture}[scale=.6]

    \draw (0,4)--(8,4);
    \draw (4,0)--(4,8);

\draw (3.6, 7.8) node  {$A$};
\draw (4.6,7.8) node  {$A^*$};
\draw (.4, 4.4) node  {$B$} ;
\draw (.4,3.6) node  {$B^*$} ;
\draw (6,6) node  {$A^*\cap B$};
\draw (2,6) node  {$A\cap B$};
\draw (2,2) node  {$A\cap B^*$};
\draw (6,2) node  {$A^*\cap B^*$};

    \draw (10,5)--(13,5)-- (13, 8);
       \draw (10,3)--(13,3)-- (13, 0);
\draw (18,5)--(15,5)-- (15, 8);
       \draw (18,3)--(15,3)-- (15, 0);

   % \draw (14,0)--(14,8);
\draw [dashed] (13,3)--(15,5) ;
\draw [dashed] (15,3)--(13,5) ;

\draw (17,6) node  {$A^*\cap B$};
\draw (11,6) node  {$A\cap B$};
\draw (11,2) node  {$A\cap B^*$};
\draw (17,2) node  {$A^*\cap B^*$};

\draw (14,7) node  {$a$};
\draw (14,1) node  {$b$};
\draw (11, 4) node  {$c$};
\draw (17, 4) node  {$d$};
\draw (13.5, 3.5) node  {$e$};
\draw (14.5, 3.5) node  {$f$};

  \end{tikzpicture}

\vskip .5cm \caption{Crossing cuts}\label{Cuts}
\end{figure}

The tree we have constructed is canonically determined, i.e. we made no choices in its construction.
The fact that this is the case is more clearly demonstrated in the construction of the next section which applies to networks based on
arbitrary graphs.   In our example the map $\nu $ is injective but not surjective, as $z$ is not in the image of $\nu $.
We can contract any edge incident with $z$ to obtain a Gomory-Hu tree \cite {GH} in which the vertices are the vertices of $X$, each such
tree will depend on the choice of edge contracted.

\section {Infinite Networks}
Let $X$ be an arbitrary connected simple graph.  It is not even assumed that $X$ is locally finite.
%We define a network  $N$ to be a finite, simple, connected graph  $X$  and a map $c : EX \rightarrow \{1, 2, \dots \}$.
Let $\B X$ be the set of all edge cuts in $X$.   Thus if $A \subset VX$, then $A \in \B X$ if $\delta A$ is finite.
Here $\delta A$ is the set of edges  which have one vertex in $A$ and one in $A^*$.  If we turn $X$ into a network
in which each edge has capacity $1$, then $\B X$ is the set of cuts with finite capacity.

A ray $R$ in $X$ is an infinite sequence $x_1, x_2, \dots $ of distinct vertices such that $x_i, x_{i+1}$ are adjacent for every $i$.
If $A$ is an edge cut, and $R$ is a ray, then there exists an integer $N$ such that for $n > N$ either $x_n \in A$ or $x_n \in A^*$.
We say that $A$ separates rays $R = (x_n), R' = (x_n')$ if for $n$ large enough either $x_n \in A, x_n' \in A^*$ or $x_n \in A^*, x_n' \in A$.
We define $R\sim R'$ if they are not separated by any edge cut.     It is easy to show that $\sim $ is an equivalence relation on the set $\Phi X$ of rays in $X$.
The set $\Omega X = \Phi X/ \sim $ is the set of {edge} ends of $X$.   An edge cut $A$ separates ends $\omega , \omega '$
if it separates rays representing $\omega , \omega '$.     A cut $A$ separates an end $\omega $ and a vertex $v  \in VX$ if for any ray representing
$\omega $,   $R$ is eventually in $A$ and $v \in A^*$ or vice versa.

%A graph has   more than one (edge) end if there is a cut $A \in \B X$ such that both $A$ and $A^*$ are both infinite.

As before we turn $X$ into  a network  $N$ by considering a map $c : EX \rightarrow \{1, 2, \dots \}$.  

Our main theorem is as follows.

\begin {theo}\label {maintheoremB}  Let $N(X)$ be a network in which $X$ is an arbitrary connected graph.  For each $n >0$, there is a network $N(T_n)$ based on a tree  $T_n$ and a map $\nu : VX\cup \Omega X \rightarrow VT\cup \Omega T$, such
that  $\nu (VX ) \subset VT$ and  $\nu x = \nu y$ for any $x, y \in VX \cup \Omega X$ if and only if $x, y$ are not separated by a cut $A$ with $c(A) \leq n$. 

The network $N(T_n)$ is canonically determined and is invariant under the automorphism group of $N(X)$.

\end {theo}

\begin {proof} 
Let $\B _n X$ be the subring of $\B X$ generated by the cuts $A$ such that $c(A) < n$

We will prove Theorem \ref {maintheoremB} by showing that  there is a canonically defined nested set $\ce _n$ of generators
of $B_n X$, with the following properties:-
\begin {itemize}
\item [(i)]  If $G$ is the automorphism group of $N(X)$, then $\ce_n $ is invariant under $G$.
\item [(ii)] For each $i <j$, $\ce _i \subseteq  \ce _j$.

\end {itemize}
In order for the nested set $\ce _n$ to be the edge set of a tree, it is necessary and sufficient that it satisfies the finite interval condition, i.e.

\begin {itemize}  
\item [ ] If $C, D \in \ce $ and $C\subset D$, then there are only finitely many $E \in \ce$ such that $C\subset E \subset D$.
\end {itemize}
This  is shown in \cite {[D1]}, \cite {[DD]} and other places.

 The fact that $\ce _n$ satisfies the finite interval condition  follows from the following useful proposition.
 
 If the cut $A$ is thin both
$A$ and $A^*$ are connected.    A cut with this property is called {\it tight} by Thomassen and Woess \cite {thomassen1993}.
It is shown in \cite {[D2]} that there are only finitely many tight cuts $C$ with a fixed capacity  such that $\delta C$ contains a particular edge.
The proof of this in \cite {thomassen1993} is neater and it is reproduced here for completeness.   By replacing each edge with capacity $c(e)$, by $c(e)$ edges joining the same pair of vertices, we can assume that every edge has capacity one.
\begin {prop}\label {tight} For any $e \in EX$, there are only finitely many tight cuts $A$ with $|\delta A| = c(A) = k$ such that $e \in \delta A$.
\end {prop}
\begin {proof}  The proof is by induction on $k$.  For $k=1$ there is nothing to prove.  So assume $k \geq 1$.    We can assume that $e = xy$ is in some tight $k$-cut,  i.e. there is a cut $A$ such that $e \in \delta A$ where $\delta A$ has $k$ edges.  Hence $X - e$ has a path $P$ from $x$ to $y$.   Now every tight $k$ cut that contains $e$ also contains an edge of $P$.   By the induction hypothesis there are only finitely many tight $(k-1)$-cuts in $X-e$ containing an edge of $P$, and we are done.
\end {proof}

The fact that $\ce _n$ satisfies the finite interval condition now follows easily.  
Thus if $C, D$ are tight cuts with $C\subset D$ and $F$ is a finite subgraph of $X$ containing $\delta C\cup \delta D$ and $C\subset E \subset D$
then $\delta E$ must contain an edge of $\delta E$.   It follows from Proposition \ref{tight} that there are only finitely many cuts $E$ such that
$c(E) \leq n$ with this property.

 As before we define a cut  $A \subset VX$ to be {\it thin} if $c(A) = n$  but $A\notin B_{n-1}X$, the subring of $\B X $ generated by those
$C \in \B X $ for which $c(C) < n-1$.   As before $\B X$ is generated by thin elements.
Two cuts are said to be {\it crossing} if they are not nested.

We use induction on $n$.  We assume that $\B _{n-1}$ has a nested set $\ce _{n-1}$ of thin generators satisfying (i), (ii)  and thus  is the edge set of a tree $T = T_{n-1}$, and this tree which has  the properties given in the theorem for $n-1$.      If $v \in VX$ then $v$ determines an orientation $\cO$ of the tree $T_{n-1}$.  Thus $\cO $ is a choice of precisely one of $A$ or $A^*$ for each $A \in ET$.   One chooses $A$ if and only if $v \in A$.  But this orientation determines a vertex of $T$, as there is a vertex $\nu v$  of $T$ such that each edge is oriented to point at  $\nu v$.    Similarly a ray in $X$ determines an orientation of $ET$ that determines either a vertex or an end of $T$.    Suppose that there are elements $x,y \in VX\cup \Omega X$ which are separated by a cut $A$ with $c(A) =n$ but which are not
separated by a cut $C$ with $c(C) < n$.   In this situation $\nu x = \nu y$ is a vertex $z$ of $T$.    We now construct a connected graph
$X_z$ corresponding to $z$. 
Each oriented edge of $T$ incident with $z$ corresponds to a cut $C$ with $c(C) < n$.  If $C, D$ are distinct oriented edges with
initial vertex $z$ then $C^* \subset D$ and if  $E \in \ce  _{n-1}$ is such that $C^* \subseteq E \subseteq D$, then either $C^* = E$ or $E = D$.
Note that if $C,D$ are as above, i.e. both have initial vertex $z$, then $C^*\cap D^* = \emptyset $.  Thus if we consider the edges of $X$  that lie in $\delta C$ for some $C$ with initial vertex $z$ then there are at most two cuts $C \in \ce _{n-1}$ such that $e \in \delta C$.  In fact if we regard $e$ as an oriented edge and take
the edges of $\delta C$ as oriented with initial vertex in $C$, then each oriented edge is in at most one such set $\delta C$. 
We define a graph $X _z$ as follows.   We take $EX _z$ to be the edges which lie in  $\delta C$ for some edge $C$ of $T$ incident with $z$.
We take $VX_z$ to be the set of vertices of these edges, but we identify vertices $u,v$ if  they both lie in $C^*$ when $C$ has initial vertex $z$.

Each vertex $v$ of $X$ for which $\nu v = z$ is a vertex of $VX _z$, but there may be no such vertices.    There is another vertex of $X_z$
for each cut $C$ with initial vertex $z$ and this vertex is obtained by identifying all the vertices of $\delta C$ that are in $C^*$.

It is fairly easy to see that $X_z$ is connected.   Thus any two vertices of $X$ are joined by a path in $X$.    If $u, v$ are two vertices of $X$ that 
become vertices of $X_z$ after carrying out the identifications just described, then the path in $X$ will become a path $p$ in $X_z$ if we delete any
edges that are not in $X_z$.   Here we use the fact that $C^*$ is connected, and when $p$ enters $C^*$ at vertex $w$  it must leave $C^*$ at a vertex $w'$ that is identified with $w$ in $X_z$.

In a similar way a ray in $X_z$ corresponds to a ray in $X$.      If the ray passes through a vertex corresponding to the cut $C$, then the two incident
edges will both lie in $\delta C$.  There will be a path in $C^*$ joining the corresponding vertices before they are identified.    For each such
vertex that is visited by the ray we can add in this path to obtain a ray in $X$.    This ray will belong to an end $\omega \in \Omega X$ such that
$\nu \omega  = z$.    

\begin {lemma}\label {cross}   If $A, B$ are crossing thin cuts, with $c(A) = m, \ c(B) =n$, then  after relabelling $A$ as $A^*$ and $B$ as $B^*$  if necessary, both  $A\cap B^*, A^*\cap B$ are thin cuts with capacities  $m, n$ respectively . 
\end{lemma}    
\begin {proof}  We refer to Figure \ref{Cuts} and follow the argument in the finite case.   Suppose $m\leq n$.
After possible relabelling we can assume  $a\leq b, \ c\leq d$.   If $a < c$ then $c(A\cap B) < n$ and $c(A^*\cap B) < n$ and so 
$B$ is not thin.   If $c <a $, then $A$ is not thin.  Hence  $a = c$.  If $a < b$, then $c(A\cap B) = 2a + f  < a +e + f + b = m$, and so $c(A\cap B^*)= a + e + b  = m$ and  $f = 0$.   Also  $c(A^*\cap B) = a + e + d = n$,   and $A^*\cap B \notin \B_nX$  and so it is thin, and we are done.    
If $a = b$, then $m  = 2a + e  + f  \leq a +e + f + d = n$ and so $b\leq d$.  thus  $a= b=c \leq  d$.  If $e$ is not $0$, then $c(A\cap B) < m, c(A^*\cap B^*) < n$ and the lemma follows easily.  If $e =0$ and $f \not= 0$, then $c(A\cap B^*) < m , c(A^*\cap B) <n$ and the lemma follows if we relabel
$A$ as $A^*$.   
But if $e = f= 0$, then $c(A\cap B) = c(A\cap B^*) = m$ and $c(A^*\cap B) = c(A^*\cap B^*) = n$.      In this situation it is not possible that two adjacent 
corners of $A, B$ are not thin.
   Thus for one pair of opposite corners we have that both corners are thin. By relabelling we can assume these
corners are $A\cap B^*$ and $A^*\cap B$  and the lemma is proved.   

\end {proof}

  We show now  that  for any cut $B$, that  there are only finitely many $A\in \ce _{n-1}$ with which $B$ is not nested.  
  \begin{lemma} Let $B$ be a cut.      There are only finitely many thin cuts $A$ with capacity $n$ that cross $B$. 
\end {lemma}
\begin {proof}    Let $F$ be a finite connected  subgraph of $X$  that contains $\delta B$.   If 
$A$ crosses $B$, then both $A$ and $A^*$ contain a vertex of $\delta B$.   Hence $F$ contains an edge of $\delta A$.   The lemma follows from
Proposition \ref {tight}.

\end{proof} 
Suppose that $\B _{n-1}$ has a canonically determined nested set of generators $\ce _{n-1}$ invariant under $G$.

If $\ce _{n-1}$ does not generate $\B _nX$, then there is a thin cut $A \in \C_n - \B _{n-1}$.

We will show that $\B _n$ is generated by a set $\ce _{n-1} \cup \C _n'$, where $\C _n'$ is the set of cuts $A \in \C _n - \B_{n-1}$ that are nested with
every $C \in \ce _{n-1}$.    

To see this, let $A \in \C _n - \B_{n-1}$.   If $A$ is not nested with some $C \in \ce _{n-1}$, then all four corners of $A, C$ are not empty.
We refer to Figure \ref{Cuts}  .    
 By relabelling $A^*$ as $A$ and $C^*$ as $C$ if necessary, then by Lemma \ref {cross} we have that $A\cap C^*$ is thin with capacity $n$
 and $A^*\cap C$ is thin with capacity $m$.

Since $B $ and $A$ are not nested we have   $\mu (A, \ce _{n-1}) \not= 0$, and  by Lemma  we have $\mu (A\cap B, \ce _{n-1})+\mu (A\cap B^*, \ce _{n-1}) < \mu (A, \ce _{n-1})$.  Thus $A = A\cap B + A\cap B^*$ and both $A\cap B$ and $A\cap B^*$ are not nested with fewer cuts in $\ce _{n-1}$  than $A$. 
 An easy induction argument now shows that $\B _n$ is generated by the set of cuts  $\ce _{n-1} \cup \C _n'$, where $\C _n'$ is the set of cuts $A \in C_n - \B_{n-1}$ that are nested with
every $C \in \ce _{n-1}$. 

A   cut  $A$  is  defined to be  {\it thin   with respect to }$u,v \in VX\cup \Omega X$ if it separates some $x, y   \in VX\cup \Omega X$ and $c(A)$ is minimal 
 among all the cuts that separate $u$ and $v $.   
 
 We will show that $A$ is thin in the earlier sense if and only if it is thin with respect to some  
 $u,v \in VX\cup \Omega X$.
 
 Let $A$ be a cut with $c(A) = n$ and suppose $A$  is nested with every $C \in \ce _{n-1}$, and suppose $A \notin \B _{n-1}$.   We obtain an orientation  $\cO$  of the edges of
 $T$ as follows.    We put $C\in \cO$ if either $C \subset A$ or $C\subset A^*$.   This orientation determines a vertex $z$ of $T$.
 Also $A$ determines a cut $A_z$ in $X_z$.    This is because the vertices of $X_z$ correspond either to a vertex of $X$ or to a cut $C \in ET$ with initial vertex $z$.   But in the latter case each such cut  is nested with $A$ and either $C^* \subset A$ or $C^* \subset A^*$.  The vertices corresponding to those $C$ such that $C^* \subset A$ together with those vertices $v$  of $A$ for which $\nu v = z$ 
 will give 
 $A_z$.    The set $A_z$ cannot consist of finitely many vertices of this latter type, as this would mean that $A \in \B _{n-1}$.
 \begin {prop} \label {thin} Let  $A$ be a cut  such that $c(A) =n$ which is nested with every $C \in ET _{n-1}$.  Then $A \notin \B_{n-1}$  if and only if there exists  $x, y  \in VX\cup \Omega X$  with respect to which $A$ is thin.
 \end {prop} 
 \begin {proof} 
 The argument above shows that neither $A_z$ nor $A_z^*$ can consist of finitely many vertices of degree less than $n$.
 Either $A_z$ contains a vertex in the image of $\nu $ or it is locally finite and infinite.   In the latter case $A_z$ will contain a ray.  We have seen above that a ray in $X_z$ can be expanded to a ray in $X$ and this ray will belong to an end $x $ such that $\nu x = z$.  Similarly there exists 
 $y \in VX \cup \Omega X$ such that $\nu y$ is either a vertex or end of $A^*_z$.

 Suppose $A$ separates $x, y $.    If $A \in \B _{n-1}$,  then $A$  can be written as a finite union of intersections of cuts 
 $B_1, B_2, \dots , B_k$ with $c(B_i) < n, i = 1, 2, \dots k$.   If no $B_i$ separates $x,y$ then neither will $A$.  Since $A$ separates $x,y$,
 it follows that
 $x,y $ are separated by some $B_i$.    If  $A$ is $n$-thin with respect to  $x,y$, then $A \notin \B_{n-1}$.

  Suppose $c(A) = n$ is nested with every $C \in ET$ and  $A\notin \B _{n-1}$. 
  We have seen that there is a vertex $z \in ET$ such that  $A$ corresponds to a cut $A_z \in \B X_z$.    Also $A \in \B _{n-1}X$ if either
  $A_z$ or $A_z^*$ consists of finitely many vertices of degree less than $n$.     Thus if $A \notin \B _{n-1}$, each of $A _z $ and $A _z^*$ either
  contains a vertex $\nu X$ or it is infinite, connected and every vertex has degree less than $n$.   But if $A_z$ has the latter property then it
  contains an infinite ray which determines a ray in $X$ belonging to an end $\omega $ such that $\nu \omega = z$.

  \end{proof}

     For $x, y \in VX \cup \Omega X$ let $\C '(x, y)$ let be the subset of $\C _n'$ consisting of those cuts which are thin with respect to $x, y$.   For $A \in  \C _n'$, $\mu (A, \C_n')$ is the finite number of cuts in $\C _n'$ with which it is not nested.
Let $\C _n'' (x, y)$ be the subset of $\C _n'(x, y)$ consisting of those $A$ for which $\mu (A) = \mu (A, \C _n')$ takes its minimal value.
If $A \in \C _n''(x,y)$, we say that $A$ is optimally nested with respect to $x, y$.
Let $\C _n''$ be the union of all the sets $C _n''(x,y)$.   
\begin {lemma}  The set $C _n''$ is nested.
\end {lemma}
\begin {proof} This is an argument from \cite {[DW]}.   Suppose $C$ is optimally nested  with respect
to  $x_1, x_2$ and $D$ is optimally nested with respect to $x_3, x_4$.   Here the $x_i$'s are elements of $VX \cup \Omega X$.  Suppose $C, D$ are not nested.   Each $x_i$ determines a corner  $A_i$ of
$C, D$.   There are two possibilities.
\begin {itemize}
\item [(i)]  The sets $x_1, x_2$ determine opposite corners, and $x_3, x_4$ determine the other two corners.
\item [(ii)]  There is a pair  of  opposite corners such that one corner is determined by one of $x_1, x_2$ and the opposite corner
is determined by one of $Y_3, Y_4$.  
\end {itemize}

In case (i) $C$ and $D$ separate both pairs $x_1, x_2$ and $x_3, x_4$.   Since $C, D$ are optimally nested with respect to
$x_1, x_2$ and $x_3, x_4$,  we have   $\mu (C) = \mu(D)$.   But now  $A_1, A_2$ are opposite corners, and so $\mu (A_1) +\mu (A_2) < \mu (C)+\mu (D) =2\mu (C)$, by Lemma \ref {corner},  since if an element of $\C _n'$ is not nested with both $A_1$ and $A_2$ it is not nested with both
$C$ and $D$ and if it is not nested with one of $A_1, A_2$ then it is not nested with one of $C$ and $D$.   The strict equality follows because 
$C \in \C _n'$ separates $x,y$ and is not nested with $D$ but both $A_1, A_2$ are nested  with $C$ and $D$.
Since both $A_1, A_2$ separate $x_1$ and $x_2$ we have a contradiction.

In case (ii) suppose these corners are $A_1 = C\cap D$ and $A_3 = C^*\cap D^*$, and that $y_1 \in A_1,  y_3 \in  A_3$.    But then $A_1$ separates $y_1$ and $y_2$ and $A_3$ separates $y_3$ and $y_4$.    Since $C$ is optimally nested with respect to $Y_1$ and $Y_2$ we have $\mu (A_1) \geq \mu (C)$ and since $D$ is optimally nested with respect to 
$y_3$ and $y_4$ we have $\mu (A_3) \geq \mu (D)$.   But it follows from  Lemma  \ref {corners_equality}
 that $\mu (A_1) + \mu(A_3) < \mu (C) + \mu(D)$ and so we have a
contradiction.    Thus $\C ''$ is a nested set  and the proof is complete.

\end {proof}   
To complete the proof of Theorem \ref {maintheoremB} it remains to show that $\B _nX$ is generated by $\ce _n = \ce _{n-1}\cup \C _n''$.

We know that $\B _nX$ is generated the thin cuts wtih capacity at most $n$.    Here we can use either definition of thin by Proposition \ref {thin}.
But we know that any two elements of $\Omega X \cup VX$ that can be separated by a cut of capacity at most $n$ can be separated by a cut in
$\ce _n$.    This means that for every $z \in VT_n$ the graph $X_z$  every cut $A_z$ either $A_Z$ or $A_z^*$ consists of finitely many vertices corresponding to cuts in $\C _n''$.   Thus $\ce _n$ generates $B_nX$.   This completes the proof of the main theorem.
\end {proof}

We illustrate the theorem with two fairly easy examples.

\eject
\begin {exam}
In this example, the graph $X$ is an infinite ladder. Every edge has capacity one.    The two ends of $X$ and some vertices are separated in $T_2$ and all ends and vertices in $T_3$.
The vertices of $T_3$ on the red central line are not in the image of $\nu$.
\end {exam}

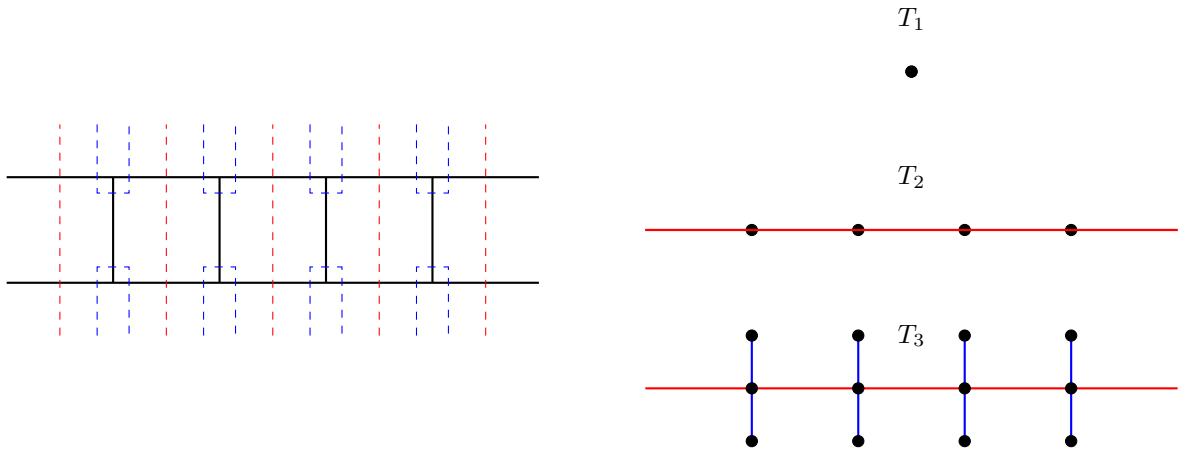
\begin{figure}[htbp]
\centering
\begin{tikzpicture}[scale=.7]
\draw [thick]  (0,0) -- (10, 0)  ;
\draw [thick](0,2) -- (10, 2)  ;
\draw [thick, red]  (12,1) -- (22, 1)  ;
\draw [thick] (2,0) --(2,2) ;
\draw [thick] (4,0) --(4,2) ;
\draw [thick](6,0) --(6,2) ;
\draw [thick] (8,0) --(8,2) ;
\draw [dashed, red] (7,-1) --(7,3) ;
\draw [dashed, red] (5,-1) --(5,3) ;
\draw [dashed, red] (3,-1) --(3,3) ;
\draw [dashed, red] (1,-1) --(1,3) ;
\draw [dashed, red] (9,-1) --(9,3) ;

\draw [dashed, blue ] (1.7,3) --(1.7,1.7)--( 2.3, 1.7) -- (2.3, 3) ;
\draw [dashed, blue ] (3.7,3) --(3.7,1.7)--( 4.3, 1.7) -- (4.3, 3) ;
\draw [dashed, blue ] (5.7,3) --(5.7,1.7)--( 6.3, 1.7) -- (6.3, 3) ;
\draw [dashed, blue ] (7.7,3) --(7.7,1.7)--( 8.3, 1.7) -- (8.3, 3) ;

\draw [dashed, blue ] (1.7,-1) --(1.7,.3)--( 2.3, .3) -- (2.3, -1) ;
\draw [dashed, blue ] (3.7,-1) --(3.7,.3)--( 4.3, .3) -- (4.3, -1) ;
\draw [dashed, blue ] (5.7,-1) --(5.7,.3)--( 6.3, .3) -- (6.3, -1) ;
\draw [dashed, blue ] (7.7,-1) --(7.7,.3)--( 8.3, .3) -- (8.3, -1) ;
\filldraw (17, 4) circle (3pt);
\draw (17,5) node  {$T_1$};
\draw (17,2) node  {$T_2$};
\draw (17,-1) node  {$T_3$};
\filldraw (17, 4) circle (3pt);
\filldraw (14, 1) circle (3pt);
\filldraw (16,1) circle (3pt);
\filldraw (18,1) circle (3pt);
\filldraw (20,1) circle (3pt);
\draw [thick, red]  (12,1) -- (22, 1)  ;
\draw [thick, red]  (12,-2) -- (22, -2)  ;
\draw [thick,blue] (14,-3) --(14,-1) ;

\draw [thick,blue] (16,-3) --(16,-1) ;
\draw [thick,blue] (18,-3) --(18,-1) ;
\draw [thick,blue] (20,-3) --(20,-1) ;
\filldraw (14, -1) circle (3pt);
\filldraw (16,-1) circle (3pt);
\filldraw (18,-1) circle (3pt);
\filldraw (20,-1) circle (3pt);

\filldraw (14, -2) circle (3pt);
\filldraw (16,-2) circle (3pt);
\filldraw (18,-2) circle (3pt);
\filldraw (20,-2) circle (3pt);
\filldraw (14, -3) circle (3pt);
\filldraw (16,-3) circle (3pt);
\filldraw (18,-3) circle (3pt);
\filldraw (20,-3) circle (3pt);

\end{tikzpicture}
\end{figure}

\begin {exam}
In this example  as shown in the diagram each edge again has capacity one.  All the  vertices are separated in $T_4$ and all ends and vertices in $T_5$.  There are vertices of infinite degree in $T_3$ and $T_4$.
%The vertices of $T_3$ on the red central line are not in the image of $\nu$.
\end {exam}

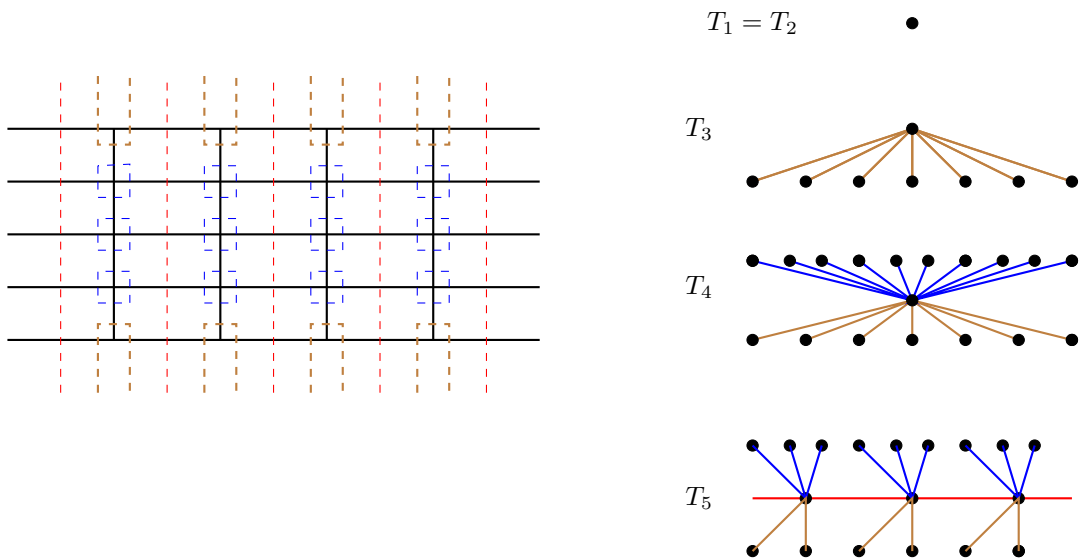
\begin{figure}[htbp]
\centering
\begin{tikzpicture}[scale=.7]
\draw [thick]  (0,1) -- (10, 1)  ;
\draw [thick]  (0,3) -- (10, 3)  ;
\draw [thick]  (0,-1) -- (10, -1)  ;

\draw [thick]  (0,0) -- (10, 0)  ;
\draw [thick](0,2) -- (10, 2)  ;
%\draw [thick, red]  (12,1) -- (22, 1)  ;
\draw [thick] (2,-1) --(2,3) ;
\draw [thick] (4,-1) --(4,3) ;
\draw [thick](6,-1) --(6,3) ;
\draw [thick] (8,-1) --(8,3) ;
\draw [dashed, red] (7,-2) --(7,4) ;
\draw [dashed, red] (5,-2) --(5,4) ;
\draw [dashed, red] (3,-2) --(3,4) ;
\draw [dashed, red] (1,-2) --(1,4) ;
\draw [dashed, red] (9,-2) --(9,4) ;

\draw [dashed, blue ] (1.7,2.3) --(1.7,1.7)--( 2.3, 1.7) -- (2.3, 2.33)--cycle ;
\draw [dashed, blue ] (3.7,2.3) --(3.7,1.7)--( 4.3, 1.7) -- (4.3, 2.3)--cycle ;
\draw [dashed, blue ] (5.7,2.3) --(5.7,1.7)--( 6.3, 1.7) -- (6.3, 2.3)--cycle ;
\draw [dashed, blue ] (7.7,2.3) --(7.7,1.7)--( 8.3, 1.7) -- (8.3, 2.3)--cycle ;

\draw [dashed, blue ] (1.7,-.3) --(1.7,.3)--( 2.3, .3) -- (2.3, -.3)--cycle ;
\draw [dashed, blue ] (3.7,-.3) --(3.7,.3)--( 4.3, .3) -- (4.3, -.3) --cycle;
\draw [dashed, blue ] (5.7,-.3) --(5.7,.3)--( 6.3, .3) -- (6.3, -.3)--cycle ;
\draw [dashed, blue ] (7.7,-.3) --(7.7,.3)--( 8.3, .3) -- (8.3, -.3) --cycle;

\draw [dashed, blue ] (1.7,1.3) --(1.7,.7)--( 2.3, .7) -- (2.3, 1.3)--cycle ;
\draw [dashed, blue ] (3.7,1.3) --(3.7,.7)--( 4.3, .7) -- (4.3, 1.3) --cycle;
\draw [dashed, blue ] (5.7,1.3) --(5.7,.7)--( 6.3, .7) -- (6.3, 1.3)--cycle ;
\draw [dashed, blue ] (7.7,1.3) --(7.7,.7)--( 8.3, .7) -- (8.3, 1.3) --cycle;

\draw [dashed, thick, brown ] (1.7,4) --(1.7,2.7)--( 2.3, 2.7) -- (2.3, 4) ;

\draw [dashed, thick, brown ] (3.7,4) --(3.7,2.7)--( 4.3, 2.7) -- (4.3, 4) ;
\draw [dashed, thick, brown ] (5.7,4) --(5.7,2.7)--( 6.3, 2.7) -- (6.3, 4) ;
\draw [dashed, thick, brown ] (7.7,4) --(7.7,2.7)--( 8.3, 2.7) -- (8.3, 4) ;

\draw [dashed, thick, brown ] (1.7,-2) --(1.7,-.7)--( 2.3, -.7) -- (2.3, -2) ;

\draw [dashed, thick, brown ] (3.7,-2) --(3.7,-.7)--( 4.3, -.7) -- (4.3, -2) ;
\draw [dashed, thick, brown ] (5.7,-2) --(5.7,-.7)--( 6.3, -.7) -- (6.3, -2) ;
\draw [dashed, thick, brown ] (7.7,-2) --(7.7,-.7)--( 8.3, -.7) -- (8.3, -2) ;

\draw [thick, brown]   (14,2) -- (17, 3) --(15,2) ;
\draw [thick, brown]   (16,2) -- (17, 3) --(17,2) ;
\draw [thick, brown]   (18,2) -- (17, 3) --(19,2) ;
\draw [thick, brown]   (20,2) -- (17, 3) ;

\draw [thick, brown]   (14,-1) -- (17, -.25) --(15,-1) ;
\draw [thick, brown]   (16,-1) -- (17, -.25) --(17,-1) ;
\draw [thick, brown]   (18,-1) -- (17, -.25) --(19,-1) ;
\draw [thick, brown]   (20,-1) -- (17, -.25) ;

\draw [thick, blue]   (14,.5) -- (17, -.25) --(14.7,.5) ;
\draw [thick, blue]   (15.3,.5) -- (17, -.25) --(16,.5) ;
\draw [thick, blue]   (16.7,.5) -- (17, -.25) --(17.3,.5) ;
\draw [thick, blue]   (18,.5) -- (17, -.25) --(18.7,.5) ;
\draw [thick, blue]   (19.3,.5) -- (17, -.25) --(20,.5) ;

\draw [thick, brown]   (14,2) -- (17, 3) --(15,2) ;
\draw [thick, brown]   (16,2) -- (17, 3) --(17,2) ;
\draw [thick, brown]   (18,2) -- (17, 3) --(19,2) ;
\draw [thick, brown]   (20,2) -- (17, 3) ;

\filldraw (17, 3) circle (3pt);
\draw (14,5) node  {$T_1= T_2$};
\draw (13,3) node  {$T_3$};
\draw (13,0) node  {$T_4$};
\draw (13,-4) node  {$T_5$};

\filldraw (17, 5) circle (3pt);
\filldraw (14, 2) circle (3pt);
\filldraw (16,2) circle (3pt);
\filldraw (18,2) circle (3pt);
\filldraw (20,-1) circle (3pt);
%\draw [thick, red]  (12,1) -- (22, 1)  ;
%\draw [thick, red]  (12,-2) -- (22, -2)  ;
%\draw [thick,blue] (14,-3) --(14,-1) ;
\filldraw (17, -.25) circle (3pt);
\filldraw (20, 2) circle (3pt);

\filldraw (15, 2) circle (3pt);
\filldraw (17,2) circle (3pt);
\filldraw (19,2) circle (3pt);
\filldraw (19,-1) circle (3pt);

%\draw [thick,blue] (16,-3) --(16,-1) ;
%\draw [thick,blue] (18,-3) --(18,-1) ;
%\draw [thick,blue] (20,-3) --(20,-1) ;
\filldraw (14, -1) circle (3pt);
\filldraw (16,-1) circle (3pt);
\filldraw (18,-1) circle (3pt);
\filldraw (15, -1) circle (3pt);
\filldraw (17,-1) circle (3pt);
\filldraw (19,-5) circle (3pt);
\filldraw (14, -5) circle (3pt);
\filldraw (16,-5) circle (3pt);
\filldraw (18,-5) circle (3pt);
\filldraw (15, -5) circle (3pt);
\filldraw (17,-5) circle (3pt);
\filldraw (19,-5) circle (3pt);

\filldraw (20,-1) circle (3pt);

\filldraw (14, -3) circle (3pt);
\filldraw (16,-3) circle (3pt);
\filldraw (18,-3) circle (3pt);
%\filldraw (20,-3) circle (3pt);
\draw [thick,  red]  (14,-4)--(20,-4) ;

\filldraw (15, -4) circle (3pt);
\filldraw (17,-4) circle (3pt);
\filldraw (19,-4) circle (3pt);
%\filldraw (21,-4) circle (3pt);

\filldraw (14, .5) circle (3pt);
\filldraw (16,.5) circle (3pt);
\filldraw (18,.5) circle (3pt);
\filldraw (20,.5) circle (3pt);

\filldraw (14.7, .5) circle (3pt);
\filldraw (16.7,.5) circle (3pt);
\filldraw (18.7,.5) circle (3pt);
%\filldraw (20,.5) circle (3pt);

\filldraw (15.3, -3) circle (3pt);
\filldraw (17.3,.-3) circle (3pt);
\filldraw (19.3,-3) circle (3pt);
%filldraw (20,.5) circle (3pt);

\filldraw (14.7, -3) circle (3pt);
\filldraw (16.7,-3) circle (3pt);
\filldraw (18.7,-3) circle (3pt);
%\filldraw (20,.5) circle (3pt);

\filldraw (15.3, .5) circle (3pt);
\filldraw (17.3,.5) circle (3pt);
\filldraw (19.3,.5) circle (3pt);
%filldraw (20,.5) circle (3pt);
\draw [thick, blue] (14, -3) --(15,-4) -- (15.3, -3) ;
\draw [thick, blue] (16, -3) --(17,-4) -- (17.3, -3) ;
\draw [thick, blue] (18, -3) --(19,-4) -- (19.3, -3) ;
\draw [thick, blue] (14.7, -3) --(15,-4)  ;
\draw [thick, blue] (16.7, -3) --(17,-4)  ;
\draw [thick, blue] (18.7, -3) --(19,-4)  ;
\filldraw (14, .5) circle (3pt);
\filldraw (16,.5) circle (3pt);
\filldraw (18,.5) circle (3pt);
\filldraw (20,.5) circle (3pt);

\draw [thick, brown]   (14,-5) -- (15,-4) --(15,-5) ;
\draw [thick, brown]   (16,-5) -- (17, -4) --(17,-5) ;
\draw [thick, brown]   (18,-5) -- (19, -4) --(19,-5) ;

\end{tikzpicture}
\end{figure}

\eject
Let $X, N$ be as before.
 For $s, t \in VX\cup \Omega X$ an {\it $(s, t)$-flow } in $N$ is a map  $f :  EX \rightarrow \{ 0, 1, 2, \dots \}$  together with an assignment of a direction to each edge  $e$ for which $f(e) \not= 0$  so that its vertices are $\iota e$ and $\tau e$ and the following holds.

\begin {itemize}
\item [(i)]  For each $v \in VX$ there are only finitely many incident edges $e$ for which $f(e) \not= 0$.
\item [(ii)] If  $f^+(v) = \Sigma ( f(e) | \iota e = v)$ and $f^-(v) = \Sigma (f(e) | \tau e =v) $, then $f^+(v) =f^-(v)$ for every $v \not= s, t$.

\item [(ii)]    For every cut $A$ that does not separate $s,t$   If we put $f^+ (A) =  \Sigma  ( f(e) | e\in \delta A,  \iota e \in A) $ and $f^-(A) = \Sigma (f (e) | e\in \delta A,  \iota e \in A^*)$, then  we have $f^+(A) = f^-(A)$. 
That is, for  every cut that does not separate $s, t$ , the flow into the cut is the same as the flow out.

\end {itemize}  
\begin {prop} For any $(s, t)$-flow and any cut $A$ such that $s\in A, t\in A^*$, the value  $f^+(A)  -  f^-(A)$ does not depend on $A$.  This value 
is denoted $|f|$.
\end {prop}
\begin {proof}
Let $A, B$ be cuts separating $s, t$.    Because $A\cap B$ also separates $s, t$, it  suffices to prove that $f^+(A)  -  f^-(A) = f^+(B)  -  f^-(B)$
when $A\subset B$.    Let $e \in \delta A$.  Either $e \in \delta B$ or $e \in \delta (B\cap A^*)$.   If $e' \in \delta B$ is not in $\delta A$ then $e' \in \delta 
(B\cap A^*$ and $\delta (B\cap A^*)$ partitions into those edges with both vertices in $A$ and those with both vertices not in $A$.
Since $A^*\cap B$ does not separate $s,t$,  $f^+(A^*\cap B) = f^-(A^*\cap B)$  and the value of $f^+ -f^-$ on the edges of $\delta (A^*\cap B)$
that are in $A$ is minus the value on the edges not in $A$.    The symmetric difference of $\delta A$ and $\delta B$ consists of the edges
in $\delta (A^*\cap B)$ and it follows that $f^+(A)  -  f^-(A) = f^+(B)  -  f^-(B)$.

\end {proof}

\begin {theo} [MFMC]  Let $N$ be a network based on a graph $X$.   Let $s, t \in VX \cup \Omega X$.   The maximum value of an  $(s, t)$-flow is the minimal capacity of a cut separating $s$ and $t$.
\end {theo}
\begin {proof}

Let $n$ be the minimal capacity of a cut in $X$ separating $s, t$.  In the structure tree $T =T_n$  there is a flow from $\nu s$ to $\nu t$ with the property that the value of the flow is $n$.  In  $T$ an end corresponds to a set of  rays.  For any vertex of $T$ there is a unique ray starting at that vertex and belonging to the end.    If both $\nu s$ and $\nu t$ are vertices then they are joined by a unique finite geodesic path.   If only one of $\nu s$ and $\nu t$ is a vertex (say $\nu s$), then there is a  unique ray starting at $\nu s$ and representing $\nu t$.  While if both $\nu s$ and $\nu t$ are ends, then there is a unique two ended path  in $T$ whose ends belong to $\nu s$ and $\nu t$.    In each case we get a flow in $T$ by assigning a constant value on the directed edges of the path.  We have to show that each such  flow corresponds to a flow in $X$.   

If $s, t \in VX$, then this follows from the usual proof of the theorem, which we repeat here.

Suppose we have an (s, t)-flow $f$ in $N$.   Let $e_1, e_2, \dots e_k$ be a path $p$ joining $s$ and $t$ with the following
property.    Each edge $e_i$ is given an orientation in the flow $f$.   This orientation will not usually be the same as that of going from
$s$ to $t$.    We say that $p$ is an $f$-augmenting path if for each $e_i$ for which $\iota e $ is $s$ or  a vertex of $e_{i-1}$ we have $f(e_i) < c(e_i)$,
and for each edge $e_i$ for which $\iota e = t$ or $\iota e $ is a vertex of $e_{i+1}$ we have $f(e_i) \not= 0$.   
For any flow augmenting path $p$ we get a new flow $f^*$ as follows.   

\begin {itemize} 

\item [(i)]  If $e \in EX$ is not in the path $p$, then $f^*(e) = f(e)$. 
\item [(ii)]  If $e$ is in the path $p$ and $f(e) =0$, then orient 
$e$ so that $\iota e $ is $s$ or  a vertex of $e_{i-1}$, and put  $f^*(e) = 1$.  Recall that we are assuming that $c(e) \not= 0$ for every $e \in EX$ and 
so we have $f^*(e) \leq c(e)$.
\item [(iii)] If $e$ is in the path $p$ and $\iota e $ is $t$ or  a vertex of $e_{i+1}$ and  $f(e) \not= 0$ , then $f^*(e) = f(e) -1$.
\item [(iv)] If $e$ is in the path $p$ and $\iota e $ is $s$ or  a vertex of $e_{i-1}$, then $f^*(e) = f(e) +1$.

\end{itemize}

The effect of changing $f$ to $f^*$ is to increase the flow along the path $p$.  We have $|f^*| = |f| +1$.

Let $S_f \subset VX$ be the set of vertices  that can be joined to $s$ by a flow augmenting path.  If $t \in S_f$, then we can use the flow augmenting
path joining $s$ and $t$ to get a new flow $f^*$.     We keep repeating this process until we obtain a flow $f$ for which $S_f$ does not contain $t$.
But now $S_f$ is a cut separating $s$ and $t$.   Also if $e \in \delta S_f$, then we have $\iota e \in S_f$ and $f(e) = c(e)$, since otherwise we can
extend the $f$-augmenting path from $s$ to $\iota e$ to an $f$-augmenting path to $\tau e$.    Thus $|f| = c(S_f)$.
But $n$ is the minimal capacity of a cut separating $s$ and $t$ and so $|f| \geq n$.   But also $|f|$ must be less than the capacity of any cut
separating $s$ and $t$ and so $|f| = n$, and  $S_f$ is a minimal cut separating $s$ and $t$.   

If $s \in VX$ and $t \in \Omega X$, then we can build up a flow from $s$ to $t$ in the following way.     Let $D$ be a cut in $ET$ separating $s$ and $t$, so that $s \in D$  and $c(D)\geq n$.  Let $X_D$ be the graph  defined as follows.
The  edge set $EX_D$ consists of all edges $e$ of $X$ which have at least one vertex in $D$, so that either $e \in \delta D$ or $e$ has both vertices
in $D$.   The vertex set $VX_D$ consists of the vertices of $EX_D$, except that we identify all such vertices that are in $D^*$.     Let this vertex be denoted $d^*$.    Thus in $X_D$ the edges incident with $D^*$ are the edges of $\delta D$.      Since $c(D) \geq n$, then as in  the previous case there is a
flow $f_D$ from $s$ to $d^*$ such that $|f_D| = n$.     Let $X_{D^*}$ be the graph defined as for $X_D$, using $D^*$ instead of $D$.   We now have a vertex $d \in VX_{C^*}$ whose incident edges are the edges of $\delta D$.       Now choose another edge $E \not= D$, such that $E^* \subset D^*$.
Thus $s \in E$.       Now form a graph $X(D^*, E)$ whose edge set consists of those edges that have at least one vertex in $D^*\cap E$ and whose 
vertex set is the set of vertices of the set of edges except that we identify the vertices that are in $D$ and also identify the vertices that are in $E^*$.
Thus in $X(D^*, E)$ there is a vertex $d$ whose incident edges are those of $\delta D$ and a vertex $e^*$ whose incident edges are those of 
$\delta E$.  The flow $f$ already constructed takes certain values on the edges of $\delta D$.   We can find a $(d, e^*)$ flow which takes these
same values on $\delta D$.    This flow together with the original flow will give an $(s, e^*)$-flow also denoted $f$ such that $|f| = n$.
We can keep on repeating this process and obtain the required $(s, t)$-flow.

 If $s, t$ are both in $\Omega X$,  choose a minimal cut $M$ separating $s$ and $t$, so that $s \in M, t \in M^*$.
Let $X_s$ be the graph  defined as follows.
The  edge set $EX_s$ consists of all edges $e$ of $X$ which have at least one vertex in $M$, so that either $e \in \delta M$ or $e$ has both vertices
in $M$.   The vertex set $VX_s$ consists of the vertices of $EX_s$, except that we identify all such vertices that are in $M^*$.     Let this vertex be denoted $m_t$.    Thus in $X_s$ the edges incident with $m_t$ are the edges of $\delta M$.      If $c(M) =n$, then by the previous case there is a
flow $f_s$ from $s$ to $m_t$ such that $|f_s| = n$.     If we carry out a similar construction for $M^*$ we obtain a flow $f_t$ from $m_s$ to $t$ with
$|f_t| = n$.   We can then piece these flows together to obtain a flow in $X$ from $s$ to $t$ with $|f| = n$.
\end {proof}
The following interesting fact emerges from the above proof in the case when $s, t \in VX$.
If $s,t \in VX$, the cuts in $C \in ET_n$ such that $s \in C, t \in C^*$ form a finite totally ordered set.   It is the geodesic in $ET_n$ joining $\nu s$ and $\nu t$.
Let $D$ be the smallest minimal cut with this property.     Then $S_f \subseteq D$, since for any vertex $u \in D^*$ there can be no $f$-augmenting path
joining $s$ and $u$.    But this must mean that $S_f = D$, since $S_f \in \B _nX$ which is generated by $ET_n$.
Although the maximal flow between $s, t$ is not usually unique, the smallest minimal cut separating $s, t$ is unique.
The way of obtaining $D$ by successively increasing the flow between $s$ and $t$ is obviously not a canonical process, as we choose  flow
augmenting paths  to increase the flow. 

A finitely generated group $G$ is said to have more than one end,  if a Cayley graph $X = X(G, S)$ of $G$ corresponding to a finite generating set $S$ has more
than one end.
\begin {theo} [Stallings' Theorem]   If $G$ is a finitely generated group with more than one end, then $G$ has a non-trivial action on a tree $T$
with finite edge stabilizers. 
\end {theo}
\begin {proof}  Let $n$ be the smallest integer for which there a cut $A$ such that $|\delta A| = n$ which separates two ends $s, t$.
The structure tree $T = T_n$ will have the required property.    Here we use the fact that the construction of $T$ is canonical and so is invariant 
under the action of automorphisms.    Thus the action of $G$ on $X$ gives an action on $T$.    Each edge of $T$ is a cut $C$ with $|\delta C| \leq n$.
The stabilizer of $C$ will permute the edges of $\delta C$ and will therefore be finite.    We also have to show that action is non-trivial.
We know there is a cut  $D$ in $ET$ that separates a pair of ends.  For such a cut both $D$ and $D^*$ are infinite.    The action of $G$ on $X$ is
vertex transitive.    There exists $g \in G$ such that the vertices of  $g\delta D $ are contained in $D$ and an element $h\in G$ such that the vertices
of $h\delta D$ are contained in $D^*$.  It hen follows that for $x = g, x = h$ or $x =  gh$ we have $xD$ is a proper subset of $D$ or $D$ is a proper subset of $xD$.   It follows from elementary Bass-Serre theory that $x$ cannot fix a vertex of $T$.
\end {proof}
This proof is essentially that of \cite {Kroen2009}.

\end{document}